\documentclass[preprint,12pt]{elsarticle}
\usepackage{lineno,hyperref}
\usepackage{graphicx}
\usepackage{subfigure}
\usepackage{psfrag}
\usepackage{url}
\usepackage{stfloats}
\usepackage{amsmath}
\usepackage{cases}
\usepackage{amssymb}
\usepackage{float}
\usepackage{amsthm}
\usepackage{txfonts}
\usepackage{multirow}
\usepackage{algorithm,algorithmicx}
\usepackage{algpseudocode}
\usepackage[justification=centering]{caption}
\renewcommand{\algorithmicrequire}{\textbf{Input:}}
\renewcommand{\algorithmicensure}{\textbf{Output:}}
\usepackage{color}

\newtheorem{remark}{Remark}
\newtheorem{definition}{Definition}
\newtheorem{lemma}{Lemma}
\newtheorem{theorem}{Theorem}
\newtheorem{example}{Example}
\newtheorem{proposition}{Proposition}

\modulolinenumbers[5]

\journal{XXX}

\begin{document}

\hypersetup{hidelinks}

\begin{frontmatter}

\title{Set stabilization of Boolean control networks based on bisimulations: A dimensionality reduction approach\tnoteref{titlenote}}

\tnotetext[titlenote]{This work was supported by the National Natural Science Foundation of China (62273201, 62350037), and the Research Fund for the Taishan Scholar Project of Shandong Province of China (TSTP20221103).
}


\author[1]{Tiantian Mu}
\ead{mutiantian@mail.sdu.edu.cn}

\author[1]{Jun-e Feng\corref{mycorrespondingauthor}}
\cortext[mycorrespondingauthor]{Corresponding author}
\ead{fengjune@sdu.edu.cn}

\author[2]{Biao Wang}
\ead{wangbiao@sdu.edu.cn}

\address[1]{School of Mathematics, Shandong University, Jinan 250100, Shandong, P. R. China.}
\address[2]{School of Management, Shandong University, Jinan 250100, Shandong, P. R. China.}

\begin{abstract}
This paper exploits bisimulation relations, generated by extracting the concept of morphisms between algebraic structures, to analyze set stabilization of Boolean control networks with lower complexity.
First,
for two kinds of bisimulation relations, called as weak bisimulation and strong bisimulation relations,
a novel verification method is provided  by constructing the bisimulation matrices.
Then the comparison for set stabilization of BCNs via two kinds of bisimulation methods is presented,
 which involves the dimensionality of quotient systems and dependency of the control laws on the original system.
Moreover, the proposed method is also applied to the analysis of probabilistic Boolean control networks to establish the unified analysis framework of bisimulations.
Finally, the validity of the obtained results is verified by the practical example.
\end{abstract}




\end{frontmatter}

\section{Introduction}\label{sec1}

Logical operations are the main tools used to simulate human intelligence using computers, to perceive the environment, to acquire information and feedback results.
As a general model of logical systems, Boolean networks (BNs) \cite{Kauffman1969BNs} can effectively model various types of phenomena existing in biological systems, engineering fields and so on.
Focusing on the lack of effective tools to study the logical systems,
the semi-tensor product (STP) of matrices is created by professor Cheng and his collaborators to develop the cross-dimensional matrix theory  in \cite{Chengdaizhan2001morgan}.
Up to now,
it has achieved highly innovative results in issues of BNs,
such as reachability and controllability \cite{Lifangfei2018reach,Liangjinling2017control}, observability and detectability \cite{Guoyuqian2018obser,Liyifeng2023obser,Zhangkuizeobser2023}, stability and
stabilization \cite{Wangyong2024setstab}\cite{Zhangqiliang2020setstab}, output tracking \cite{Zhanganguo2021track}, and synchronization \cite{Zhangx2020outputtrack} \cite{Yangj2020synch}, etc.
Trivially, set stabilization is efficient to deal with the cases of observability,
stabilization, output tracking, (partial) synchronization and so on.

As a critical and fundamental issue in biological systems and engineering,
set stabilization exhibits the ability of systems to stabilize the target set from all initial states.
Existing results about set stabilization of Boolean control networks (BCNs), obtained by adding controls to BNs, are abundant and reliable,
which include but are not limited to the methods of the invariant subsets \cite{Guoyuqian2015setsta},
feedback controls \cite{Lifangfei2017setsta}, pinning controls \cite{Liurongjian2017setsta}\cite{Lixiaodong2020setsta},
sampled-data state feedback controls \cite{Zhushiyong2020setsta}, Lyapunov functions \cite{Chenbingquan2020setsta}
flipping mechanisms \cite{Duleihao2022setsta}, event-triggered controls \cite{Chenbing2023eventtrigg},
output-feedback controls \cite{Yangjian2024setsta}
and distributed controls \cite{Linlin2024setsta}.
However, all the cited techniques are restricted to the analysis of large-scale BNs.
In this regard, one promising direction is to effectively reduce the model complexity and improve simulation efficiency,
which is the scope of model reduction.

Model reduction aims to find a model with lower order and simpler structure to approximately characterize certain properties of the original system.
The related results are widely used in smart grids \cite{Lij2022modelreduction}, engineering \cite{AresdeParga2023modelreduction},
 biological system \cite{Koellermeier2024modelreduction}
 and other fields.
In the framework of STP, certain research results in model reduction of BNs have been derived,
which mainly utilize the graph theory \cite{Veliz-Cuba2014modelreduction}\cite{Motoyama2013modelreduction}
and Lyapunov theory \cite{Mengmin2016modelreduction}
to analyze topological structure, stability and other properties.
Among the bisimulation method of BCNs is also a common and reasonable method for model reduction,
aiming at realizing the mutual matching of state trajectories of the original and reduced systems
by bisimulation relations based on the state spaces of a pair of systems.
Currently, the results about bisimulation of BCNs are given in \cite{Lirui2018modelreduction}\cite{Lirui2021modelreduction}\cite{Lirui2022modelreduction}.
Among them, \cite{Lirui2018modelreduction}\cite{Lirui2021modelreduction} gave the criteria of bisimulations,
 as well as the calculation method of the maximum bisimulation relation contained in an equivalence relation.
In addition, the research scope of bisimulations is also expanded to probabilistic BCNs (PBCNs)  \cite{Lirui2022modelreduction},
BNs with impulsive effects \cite{Zhangqiliang2019modelreduction} and BNs with time delays  \cite{Yuw2023modelreduction}.
Nevertheless,
the analysis of bisimulation of BCNs and PBCNs
has still not formed a centralized framework,
due to the complex and diverse criteria.
Moreover,
there are currently several concepts of bisimulation of BCNs in \cite{Lirui2021modelreduction,Lirui2021quotients}.
As shown in \cite{Lirui2021quotients},
 it required that starting from the state pair in bisimulation relation,
the state trajectory pair can be matched mutually under any same input sequence.
But as the natural extension of the bisimulation relation proposed by \cite{Lirui2021quotients},
the bisimulation defined in \cite{Lirui2021modelreduction} does not require the condition for identical input sequence.
Unfortunately,
the existing results fall short towards to compare of the different bisimulation methods for the set stabilization problems,
 and based on which, design all kinds of feedback controllers.

In this paper,  we advocate
to investigate and compare the set stabilization of BCNs under two kinds of bisimulation methods,
as well as the design of all kinds of state feedback controllers.
The quantitative results are shown precisely as follows:

1) A new method judging two kinds of bisimulation relations
is provided by constructing bisimulation matrices,
bases on which, the maximum bisimulation relations contained in an equivalence relation is also obtained concisely, along with the minimal quotient system.

2) The comparison about set stabilization of BCNs under two kinds of bisimulation methods is presented,
which includes the dimensionality of the quotient systems,
and dependency of the control laws on the original system.

3) The proposed method has lower computational complexity in set stabilization of BCNs,
 and can also rise to the exploration of PBCNs,
thus forming a unified framework for analyzing bisimulation relations.

The rest of this work is constructed as follows. Section II introduces some notations,
definition of STP,  algebraic forms of BCNs and problem formulation.
Section III includes novel criteria of two kinds of bisimulation relations,
comparison of set stabilization of BCNs under two kinds of bisimulation methods,
as well as analysis of probabilistic bisimulations of PBCNs.
Section IV gives one practical example to verify
the effectiveness of the proposed results.
A brief conclusion is drawn in the final section.

\section{Preliminaries}

\subsection{Notations}

\begin{itemize}
\item~$\mathbb{N}:$ the set of all natural numbers.

\item~$\mathcal{D}$~$=\{0, 1\}$.

\item~$[m:n]$:~$=\{m,m+1,\ldots,n\}$, where $m,n\in\mathbb{N}$ and $m<n$.

\item~$\sim$:~two different expressions of the same thing.

\item~$\mathbb{R}^{n\times m} (\mathcal{B}^{n\times m}):$ the set of all $n\times m$ real (Boolean) matrices.

\item~${\mathrm{Col}}_{i}(A)({\mathrm{Row}}_{i}(A)):$ the $i$-th column (row) of matrix $A$.
\item~$\mathrm{Col}(A)$ $(\mathrm{Row}(A)):$ the column (row) set of matrix $A$.
\item~${[A]}_{i,j}={\mathrm{Row}}_{i}({\mathrm{Col}}_{j}(A))$.

\item~$\delta_{n}^{i}:={\mathrm{Col}}_i(I_{n})$,
$\Delta_{n}$:~$=\mathrm{Col}(I_n)$.
Simply, $\Delta$:~$=\Delta_2$.

\item~$\mathbf{1}_{n}:=\sum_{i=1}^{n}\delta_{n}^{i}$.

\item~$\mathbf{0}_{n}$: an $n$-dimensional zero column vector.

\item~$A\geq B$: ${[A]}_{i,j}\geq {[B]}_{i,j}$ for any $i$ and $j$.

\item~$A^\mathrm{T}\in\mathbb{R}^{n\times m}$: the transposition of matrix $A\in\mathbb{R}^{m\times n}$.


\item~$\delta_{n}[i_{1}, i_{2}, \cdots, i_{m}]=[\delta_{n}^{i_{1}}\ \delta_{n}^{i_{2}}\ \cdots\ \delta_{n}^{i_{m}}]$: the $n\times m$ logical matrix.

\item~$\mathcal{L}_{n\times m}$: the set of $n\times m$ logical matrices.

\item~$A\wedge B\in\mathcal{B}^{n\times m}$: ${[A\wedge B]}_{i,j}={[A]}_{i,j}\wedge {[B]}_{i,j}$,
where $A, B \in\mathcal{B}^{n\times m}$.

\item~$\mathrm{sgn}(A)\in \mathcal{B}^{n\times m}$: ${[\mathrm{sgn}(A)]}_{i,j}=1$ if ${[A]}_{i,j}\neq0$, and $0$ if ${[A]}_{i,j}=0$, where $A\in\mathbb{R}^{n\times m}$.

\item~$\mathcal{I}(S)$:~$=\sum_{i=1}^{k}\delta_n^{\alpha_i}:=\delta_n[\alpha_1\hat{+}\alpha_2\hat{+}\cdots\hat{+}\alpha_k]\in \mathcal{B}^{n\times 1}$,
where $S=\{\delta_n^{\alpha_i} |i\in[1:k]\}$.

\item~$|S|$~: the cardinality (number of the elements) of set $S$.

\item~$S^c$:~$=\Delta_n \setminus S$, where $S\subseteq \Delta_n$.

\item~$\otimes$: the Kronecker product

\item~$\ast$: the Khatri-Rao product.

\item~$\odot$: the Boolean product.

\end{itemize}

\subsection{The relevant knowledge of STP}

 The definition of STP is given firstly.
\begin{definition} {\rm\cite{Chengdaizhan2011b}}
Let $A\in \mathbb{R}^{m\times n}$, $B\in \mathbb{R}^{p\times q}$.
The STP of $A$ and $B$ is
\begin{equation}
 A\ltimes B=(A\otimes I_{s/n})(B\otimes I_{s/p}),
\end{equation}
 where $s$ is the least common multiple of $n$ and $p$.
 \end{definition}

In this paper, the default matrix product is STP, and symbol ``$\ltimes$'' is omitted,
unless otherwise specified.
Some properties of STP are presented to support the following discussion.

\begin{lemma}\label{le1}{\rm\cite{Chengdaizhan2011b}}
Let $A \in \mathbb{R}^{m\times n}$, $w\in \Delta_k$, $x\in \mathbb{R}^{m\times 1}, y\in\mathbb{R}^{n\times 1}$.
Then the following statements hold.

1) $wA=(I_k\otimes A)w.$

2) $w^2=\Phi_{k}w$, where $\Phi_{k}:=\mathrm{diag}\{\delta_k^1,\ldots,\delta_k^k\}$ is the $k$-valued power-reducing matrix.

3) $W_{[m,n]}xy=yx$,
 where $W_{[m,n]}=[I_n\otimes \delta_m^1,\ I_n\otimes \delta_m^2,\ldots, I_n\otimes \delta_m^m]$ be the swap matrix.
\end{lemma}

Each $i\in\mathcal{D}$ is mapped to $\delta_2^{2-i}\in\Delta$, i.e., $1\sim \delta_2^1$ and $0\sim \delta_2^2$.
 Based on this, the algebraic expression of a logical function is proposed as follows.

\begin{lemma}\label{Lf}{\rm\cite{Chengdaizhan2011b}}
Let $f(X_{1},X_{2},\ldots,X_{n}):\mathcal{D}^n\mapsto \mathcal{D}$ be a logical function.
Then there exists a structure matrix $L_f \in \mathcal{L}_{2\times 2^{n}}$ such that
its algebraic form is expressed as follows:
\begin{equation}
f(X_{1},X_{2},\ldots,X_{n})\sim L_f\ltimes_{i=1}^{n}x_{i},
\end{equation}
where $x_i=\delta_2^{2-X_i}\in\Delta, i\in[1:n]$.
\end{lemma}

 \begin{lemma}\label{KR}{\rm\cite{Chengdaizhan2011a}}
Suppose that
 \begin{equation*}
\left\{
\begin{array}{l}
y=M_{y}\ltimes_{i=1}^{n}x_{i},\\
z=M_{z}\ltimes_{i=1}^{n}x_{i},
\end{array}
\right.
\end{equation*}
where $y, z, x_{i}\in\Delta, i\in[1:n], M_{y}, M_{z}\in \mathcal{L}_{2\times 2^{n}}$. Then
\begin{equation*}
yz=(M_{y}\ast M_{z})\ltimes_{i=1}^{n}x_{i},
\end{equation*}
where ${\mathrm{Col}}_i(M_y\ast M_z)={\mathrm{Col}}_{i}(M_y)\otimes {\mathrm{Col}}_{i}(M_z), i\in[1:2^{n}]$.
\end{lemma}

\subsection{The logical and algebraic forms of BCNs}

Consider BCN  with $n$ state nodes $\{X_i\in\mathcal{D}| i\in [1:n]\}$ and $m$ inputs $\{U_j\in \mathcal{D}| j\in[1:m]\}$,
\begin{equation}\label{BCNlog}
\left\{
\begin{aligned}
X_{1}(t+1)&=f_1(U_1(t),\ldots,U_m(t), X_1(t),\ldots,X_n(t)),\\
X_{2}(t+1)&=f_2(U_1(t),\ldots,U_m(t), X_1(t),\ldots,X_n(t)),\\
&\vdots\\
X_{n}(t+1)&=f_n(U_1(t),\ldots,U_m(t), X_1(t),\ldots,X_n(t)), t\in\mathbb{N},
\end{aligned}
\right.
\end{equation}
where $f_i:{\mathcal{D}}^{m+n}\mapsto \mathcal{D}$, $i\in[1:n]$.
Let $x_i(t)=\delta_2^{2-X_i(t)}\in \Delta$, $u_i(t)=\delta_2^{2-U_i(t)}\in\Delta$.
 Using STP, one has $x(t)=\ltimes_{i=1}^{n}x_{i}(t)\in\Delta_{N}$ and $u(t)=\ltimes_{j=1}^{n}u_{j}(t)\in\Delta_{M}$, where $N=2^n$ and $M=2^m$.
Then there exists $F_i\in \mathcal{L}_{2\times NM}$,
such that $x_i(t+1)=F_i u(t)x(t), i\in [1:n]$.
According to Lemma \ref{KR}, the algebraic form of BCN \eqref{BCNlog} is established as follows:
\begin{equation}{\label{BCNalg}}
x(t+1)=Fu(t)x(t), t\in\mathbb{N},
\end{equation}
where $F=F_1\ast F_2\ast\cdots \ast F_n\in \mathcal{L}_{N\times NM}$
is the state transition matrix of BCN \eqref{BCNlog}.

The input feedback law of BCN \eqref{BCNalg} is expressed as
\begin{equation}\label{ualg}
    u(t)=G x(t), t\in \mathbb{N},
\end{equation}
where $G\in\mathcal{L}_{M\times N}$.

\subsection{Problem formulation}

In this paper, we aim to analyze the set stabilization of BCNs by two types of bisimulation methods.
The set stabilization of the BCN reflects its ability to reach and remain continuously on the target set from any initial state after a finite time.
Denote by $x(t,x_0, \mathbf{u})$ the state at the $t$-th step of BCN \eqref{BCNalg} from initial state $x_0\in\Delta_N$, under a control sequence $\mathbf{u}={\{u(t)\}}_{t\in\mathbb{N}}$.
Then we give the following definition.

\begin{definition}{\rm \cite{Guoyuqian2015setsta}}\label{defsetstab}
Let $\mathcal{A}\subseteq\Delta_N$ .
BCN \eqref{BCNalg} is said to be $\mathcal{A}$-stabilizable,
if for any initial state $x_0\in\Delta_N$, there
exist a control sequence $\mathbf{u}={\{u(t)\}}_{t\in\mathbb{N}}$ and $T_{x_0}\in\mathbb{N}$,
such that
$x(t,x_0, \mathbf{u})\in\mathcal{A}, \forall t\geq T_{x_0}$.
\end{definition}

In what follows, we dedicate to analyzing the set stabilization,
by two kinds of bisimulation relations,
 called as weak bisimulation and strong bisimulation relations between BCN \eqref{BCNalg} and its replication system possessing the similar state transition matrix with BCN \eqref{BCNalg}.
The concepts are proposed as follows.

\begin{definition}{\rm \cite{Lirui2021modelreduction} (Weak Bisimulation)}\label{defbi-simulation}
A relation $\mathcal{R} \subseteq \Delta_{N}\times \Delta_N$ is called the wrong bisimulation relation between BCN \eqref{BCNalg} and its replication system,
 if for all
$(x, y) \in \mathcal{R}$, we have

(i) for every $u\in \Delta_{M}$, there exists a $v\in \Delta_{M}$,
such that
\begin{equation}\label{FuxFvy}
  (F u x, F v y)\in \mathcal{R}.
\end{equation}

(ii) for every $v\in \Delta_{M}$, there exists a $u \in \Delta_{M}$,
such that again \eqref{FuxFvy} holds.
\end{definition}

\begin{definition}{\rm \cite{Lirui2021quotients} (Strong Bisimulation)}\label{defbi-simulation2}
A relation $\mathcal{R} \subseteq \Delta_{N}\times \Delta_N$ is called the strong bisimulation relation between BCN \eqref{BCNalg} and its replication system,
 if for all
$(x, y) \in \mathcal{R}$, one has
\begin{equation}\label{FuxFuy}
  (F u x, F u y)\in \mathcal{R}, \forall u\in \Delta_{M}.
\end{equation}
\end{definition}

\begin{remark}
Note that $N, M$ in Definitions \ref{defbi-simulation} and \ref{defbi-simulation2} are not only powers of 2,
but can still be applied to any value.
\end{remark}

\begin{example}\label{exam-comparsionIandII}
Consider a BCN with state transition matrix $F=\delta_8[ 2\ 3\ 1\ 5\ 6\ 7\ 8\ 5\ 3$
$ 2\ 5\ 2\ 6\ 7\ 8\ 7]$,
where $N=8$ and $M=2$.
Suppose that an equivalence relation $\mathcal{R}=\mathcal{R}^1\cup \mathcal{R}^2\cup \mathcal{R}^3$,
where
\begin{equation}
\begin{aligned}
\mathcal{R}^1&=\{(\alpha_1, \alpha_2) | \alpha_1, \alpha_2\in\{\delta_8^1, \delta_8^2\}\},\ \mathcal{R}^2=\{(\alpha_1, \alpha_2) | \alpha_1, \alpha_2\in\{\delta_8^3, \delta_8^4\}\},\\
\mathcal{R}^3&=\{(\alpha_1, \alpha_2) | \alpha_1, \alpha_2\in\{\delta_8^5, \delta_8^6, \delta_8^7, \delta_8^8\}\}.
\end{aligned}
\end{equation}
Through verification,
 $\mathcal{R}$ is the weak bisimulation relation,
but not strong bisimulation relation.
Moreover, a strong bisimulation relation contained in $\mathcal{R}$ is derived,
that is $\{(\delta_8^1, \delta_8^1)\}\cup \{(\delta_8^2, \delta_8^2)\}\cup\{(\delta_8^3, \delta_8^3)\}\cup\{(\delta_8^4, \delta_8^4)\}\cup \mathcal{R}^3$.
\end{example}

\begin{remark}
As noted in Example \ref{exam-comparsionIandII},
it is evident that compared with the strong bisimulation relation in Definition \ref{defbi-simulation2},
the weak bisimulation relation in Definition \ref{defbi-simulation} has more relaxed condition,
based on which a larger bisimulation relation can be generated.
Hence, the strong bisimulation relation implies the weak bisimulation relation.
\end{remark}

\section{Main Results}

In this section, we first construct an equivalence relation (satisfying reflexivity, symmetry and transitivity) uniquely produced by a target set for stabilization.
Then we provide a new method to establish the quotient system of the original system by the weak bisimulation relation,
along with the effective dissemination of set stabilization between the original and quotient systems.
Meanwhile, we propose the state feedback control strategy for the original system,
by gathering the information from the quotient system.
Moreover, a comparative analysis is given between weak bisimulation and strong bisimulation methods.
Finally, the proposed method is applied for the analysis of PBCNs.

For target set $\mathcal{A}$ of stablization,
construct
\begin{equation}\label{SA}
    \mathcal{R}_{\mathcal{A}}=\{(\alpha_1,\alpha_2)| \{\alpha_1, \alpha_2\}\subseteq \mathcal{A}\ \mbox{or}\
    \{\alpha_1, \alpha_2\}\subseteq \mathcal{A}^c\}.
\end{equation}
Obviously, $\mathcal{R}_{\mathcal{A}}$ is an equivalence relation on $\Delta_N$.
And we construct $\Upsilon_{\mathcal{R}_{\mathcal{A}}}\subseteq {\mathcal{B}}_{N\times N}$ as the matrix form of $\mathcal{R}_{\mathcal{A}}$,
where
\begin{equation}\label{matrixformSA}
 {[\Upsilon_{\mathcal{R}_{\mathcal{A}}}]}_{i,j} =
 \left\{\begin{aligned}
     1,\ &\mbox{if}\ (\delta_{N}^i, \delta_{N}^j)\in \mathcal{R}_{\mathcal{A}},\\
     0,\ &\mbox{otherwise}.
 \end{aligned}
 \right.
\end{equation}

\subsection{The analysis for set stabilization of BCNs via weak bisimulation}

Consider BCN  \eqref{BCNalg}.
Let $\Psi_{1}=\mathrm{sgn}(F {\mathbf{1}}_{M})\in {\mathcal{B}}_{N\times N}$ be a one-step reachability matrix of BCN \eqref{BCNalg}, where ${[\Psi_{1}]}_{i,j}=1$ means that state $\delta_{N}^i$ is reachable at the first step from state $\delta_{N}^j$.
On this basis, we first introduce the necessary condition of the weak bisimulation relation between BCN  \eqref{BCNalg} and its replication system.

\begin{proposition}\label{probi-sim}
Suppose that equivalence relation $\mathcal{R}_{\mathcal{A}}\subseteq \Delta_N\times \Delta_N$ is a weak bisimulation relation between BCN \eqref{BCNalg} and its replication system.
If $(\delta_N^i,\delta_N^j)\in\mathcal{R}_{\mathcal{A}}$, then
\begin{equation}\label{equprobi-sim}
    {\mathrm{Col}_i(\Upsilon_{\mathcal{R}_{\mathcal{A}}}\odot\Psi_{1})}
    ={\mathrm{Col}_j(\Upsilon_{\mathcal{R}_{\mathcal{A}}}\odot\Psi_{1})}.
\end{equation}
\end{proposition}

\begin{proof}
For $(\delta_N^i,\delta_N^j)\in\mathcal{R}_{\mathcal{A}}$,
denote ${\mathrm{Col}_i(\Upsilon_{\mathcal{R}_{\mathcal{A}}}\odot\Psi_{1})}:={[\alpha_1\ \cdots\  \alpha_N]}^{\mathrm{T}}\in\mathcal{B}^{N\times 1}$ and $\mathrm{Col}_j(\Upsilon_{\mathcal{R}_{\mathcal{A}}}$ $\odot\Psi_{1}):={[\beta_1\ \cdots\  \beta_N]}^{\mathrm{T}}\in\mathcal{B}^{N\times 1}$.
 Then for any $l\in[1:N]$,
 if $\alpha_l={[\Upsilon_{\mathcal{R}_{\mathcal{A}}}\odot\Psi_{1}]}_{l,i}={\vee}_{k=1}^N {[\Upsilon_{\mathcal{R}_{\mathcal{A}}}]}_{l,k}$ ${[\Psi_{1}]}_{k,i}=1$,
 then there exists $k^*\in[1:N]$,
 such that ${[\Upsilon_{\mathcal{R}_{\mathcal{A}}}]}_{l,k^*}={[\Psi_{1}]}_{k^*,i}=1$.
 It means that there exists $u\in\Delta_{M}$,
 such that $\delta_N^{k^*}=Fu\delta_N^i$ and $(\delta_N^l,\delta_N^{k^*})\in \mathcal{R}_{\mathcal{A}}$.
 By condition (i) in Definition \ref{defbi-simulation} and the transitivity of $\mathcal{R}_{\mathcal{A}}$,
  there exists $v\in\Delta_M$, such that $Fv\delta_N^j:=\delta_N^p$, $(\delta_N^{k^*},\delta_N^p)\in \mathcal{R}_{\mathcal{A}}$ and $(\delta_N^l,\delta_N^p)\in \mathcal{R}_{\mathcal{A}}$,
  which implies that ${[\Upsilon_{\mathcal{R}_{\mathcal{A}}}]}_{l,p}={[\Psi_{1}]}_{p,j}=1$, \emph{i.e.,}
$\beta_l={[\Upsilon_{\mathcal{R}_{\mathcal{A}}}\odot\Psi_{1}]}_{l,j}={\vee}_{k=1}^N {[\Upsilon_{\mathcal{R}_{\mathcal{A}}}]}_{l,k}{[\Psi_{1}]}_{k,j}=1$.
Hence, we get that $\alpha_l\leq \beta_l, \forall l\in[1:N]$,
\emph{i.e.,} ${\mathrm{Col}_i(\Upsilon_{\mathcal{R}_{\mathcal{A}}}\odot\Psi_{1})}
\leq{\mathrm{Col}_i(\Upsilon_{\mathcal{R}_{\mathcal{A}}}\odot\Psi_{1})}$.
Since an equivalence relation $\mathcal{R}_{\mathcal{A}}\in\Delta_{N}\times \Delta_{N}$ satisfies symmetry,
 then by the same procedure above, we can prove that ${\mathrm{Col}_j(\Upsilon_{\mathcal{R}_{\mathcal{A}}}\odot\Psi_{1})}\leq
 {\mathrm{Col}_i(\Upsilon_{\mathcal{R}_{\mathcal{A}}}\odot\Psi_{1})}$.
 Hence, equation \eqref{equprobi-sim} holds.
\end{proof}

According to Proposition \ref{probi-sim},
 we construct $M_{\mathcal{R}_{\mathcal{A}},F}\in\mathcal{B}^{N\times N}$, as a weak bisimulation matrix  of BCN  \eqref{BCNalg},
where
\begin{equation}\label{Ms,psi}
    {[M_{\mathcal{R}_{\mathcal{A}},F}]}_{i,j}=\left\{
    \begin{aligned}
        1,\ &\mbox{if} \ \mathrm{Col}_i(\Upsilon_{\mathcal{R}_{\mathcal{A}}}\odot\Psi_{1})
        ={\mathrm{Col}_j(\Upsilon_{\mathcal{R}_{\mathcal{A}}}\odot\Psi_{1})},\\
        0,\ &\mbox{otherwise}.
    \end{aligned}
    \right.
\end{equation}

Then the criterion for the weak bisimulation relation can be formulated.

\begin{theorem}\label{theobi-sim}
An equivalence relation $\mathcal{R}_{\mathcal{A}}\subseteq\Delta_{N}\times \Delta_{N}$ in \eqref{SA} is a weak bisimulation relation between BCN \eqref{BCNalg} and its replication system, if and only if
\begin{equation}\label{equtheobi-sim}
\Upsilon_{\mathcal{R}_{\mathcal{A}}}\leq M_{\mathcal{R}_{\mathcal{A}},F}.
\end{equation}

\end{theorem}

\begin{proof}
(Necessity):
Suppose that ${\mathcal{R}}_{\mathcal{A}}\subseteq\Delta_{N}\times \Delta_{N}$ is a weak bisimulation relation between BCN \eqref{BCNalg} and its replication system as per Definition \ref{defbi-simulation}.
For any $(\delta_{N}^i, \delta_{N}^j)\in \mathcal{R}_{\mathcal{A}}$,
one has ${[\Upsilon_{\mathcal{R}_{\mathcal{A}}}]}_{i,j}=1$.
By Proposition \ref{probi-sim},
we claim that
$\mathrm{Col}_i(\Upsilon_{\mathcal{R}_{\mathcal{A}}}\odot\Psi_{1})
=\mathrm{Col}_j(\Upsilon_{\mathcal{R}_{\mathcal{A}}}\odot\Psi_{1})$, \emph{i.e.,} ${[M_{\mathcal{R}_{\mathcal{A}},F}]}_{i,j}=1$.
Therefore, equation \eqref{equtheobi-sim} holds.

(Sufficiency):
For any $(\delta_{N}^i, \delta_{N}^j)\in \mathcal{R}_{\mathcal{A}}$,
 one has ${[\Upsilon_{\mathcal{R}_{\mathcal{A}}}]}_{i,j}=1$.
From equation \eqref{equtheobi-sim},
one can be referred that ${[M_{\mathcal{R}_{\mathcal{A}},F}]}_{i,j}=1$,
which equals that ${\mathrm{Col}_i(\Upsilon_{\mathcal{R}_{\mathcal{A}}}\odot\Psi_{1}})
={\mathrm{Col}_j(\Upsilon_{\mathcal{R}_{\mathcal{A}}}\odot\Psi_{1})}$.
For any $u\in\Delta_M$,
let $Fu\delta_N^i:=\delta_N^{k^*}$,
then one has ${[\Psi_{1}]}_{k^*,i}=1$.
On this basis,
it follows that
${[\Upsilon_{\mathcal{R}_{\mathcal{A}}}\odot\Psi_{1}]}_{k^*,i}
={[\Upsilon_{\mathcal{R}_{\mathcal{A}}}\odot\Psi_{1}]}_{k^*,j}={\vee}_{k=1}^N {[\Upsilon_{\mathcal{R}_{\mathcal{A}}}]}_{k^*,k}{[\Psi_{1}]}_{k,i}=1$.
As ${[\Upsilon_{\mathcal{R}_{\mathcal{A}}}\odot\Psi_{1}]}_{k^*,j}={\vee}_{k=1}^N {[\Upsilon_{\mathcal{R}_{\mathcal{A}}}]}_{k^*,k}{[\Psi_{1}]}_{k,j}=1$,
we can find $p^*\in[1:M]$, such that ${[\Upsilon_{\mathcal{R}_{\mathcal{A}}}]}_{k^*,p^*}={[\Psi_{1}]}_{p^*,j}=1$.
It means that there exists $v\in\Delta_M$,
such that $\delta_N^{p^*}=Fv\delta_N^j$ and $(\delta_N^{k^*},\delta_N^{p^*})=(Fu\delta_N^i, Fv\delta_N^j)\in \mathcal{R}_{\mathcal{A}}$,
which means that condition
(i) of Definition \ref{defbi-simulation} holds.
According to the symmetry of $\mathcal{R}_{\mathcal{A}}$,
 condition
(ii) of Definition \ref{defbi-simulation} is also satisfied.
The proof is completed.
\end{proof}

\begin{remark}
In fact,
the conclusion of \ref{theobi-sim} is identical to Theorem 3.1' in \cite{Lirui2021modelreduction},
represented as $\Upsilon_{\mathcal{R}_{\mathcal{A}}}\odot\Psi_{1}\geq \Psi_{1} \odot \Upsilon_{\mathcal{R}_{\mathcal{A}}}$.
Moreover,
Theorem 1 reveals the essential requirement of the weak bisimulation relation that
all states in the same equivalence class have the similar reachability to any equivalence class,
which cannot be reflected by Theorem 3.1 in \cite{Lirui2021modelreduction}.
\end{remark}

\begin{example}\label{exam1}
Consider a BCN proposed by \cite{Lirui2021modelreduction},
where
\begin{equation}
\begin{aligned}
  F=\delta_8[2\ 1\ 4 \ 5\ 6\ 7\ 8\ 4\ 1\ 1\ 8\ 5\ 6\ 7\ 8\ 4].
\end{aligned}
\end{equation}
Choose $\mathcal{R}=\{(\delta_8^{\alpha_1}, \delta_8^{\alpha_2}) |\{\alpha_1, \alpha_2\}\subseteq [1:2]\}\cup \{(\delta_8^{3}, \delta_8^{3})\}\cup\{(\delta_8^{\alpha_3}, \delta_8^{\alpha_4}) |\{\alpha_3, \alpha_4\}\subseteq [4:8]\}$.
It is obvious that $\mathcal{R}\subseteq \Delta_8\times \Delta_8$ is an equivalence relation on $\Delta_8$.
Next, we dedicate to analyzing whether $\mathcal{R}$ is a weak bisimulation relation.

First, for equivalence relation $\mathcal{R}$,
we construct its matrix form as
$\Upsilon_{\mathcal{R}}=\delta_8[1\hat{+}2,$ $ 1\hat{+}2,
3, 4\hat{+}5\hat{+}6\hat{+}7\hat{+}8, 4\hat{+}5\hat{+}6 $ $\hat{+}7\hat{+}8, 4\hat{+}5\hat{+}6\hat{+}7\hat{+}8, 4\hat{+}5\hat{+}6\hat{+}7\hat{+}8,
4\hat{+}5\hat{+}6\hat{+}7\hat{+}8]$.
By transition matrix $F$,
one has $\Psi_1=\mathrm{sgn}(F\mathbf{1}_2)=\delta_{8}[1\hat{+}2, 1, 4\hat{+}8, 5, 6,$ $7, 8, 4]$.
Then we calculate
\begin{equation}
    \begin{aligned}
     \Upsilon_{\mathcal{R}}\odot \Psi_1=\delta_8[&1\hat{+}2, 1\hat{+}2, 4\hat{+}5\hat{+}6\hat{+}7\hat{+}8, 4\hat{+}5\hat{+}6\hat{+}7\hat{+}8, 4\hat{+}5\hat{+}6\hat{+}7\hat{+}8,\\
    &4\hat{+}5\hat{+}6\hat{+}7\hat{+}8, 4\hat{+}5\hat{+}6\hat{+}7\hat{+}8, 4\hat{+}5\hat{+}6\hat{+}7\hat{+}8].
    \end{aligned}
\end{equation}
From equation \eqref{Ms,psi}, it follows that
\begin{equation*}
\begin{aligned}
    M_{\mathcal{R},F}=\delta_8[&1\hat{+}2, 1\hat{+}2, 3\hat{+}4\hat{+}5\hat{+}6\hat{+}7\hat{+}8,  3\hat{+}4\hat{+}5\hat{+}6\hat{+}7\hat{+}8, 3\hat{+}4\hat{+}5\hat{+}6\hat{+}7\hat{+}8,\\
    &3\hat{+}4\hat{+}5\hat{+}6\hat{+}7\hat{+}8, 3\hat{+}4\hat{+}5\hat{+}6\hat{+}7\hat{+}8, 3\hat{+}4\hat{+}5\hat{+}6\hat{+}7\hat{+}8]\geq \Upsilon_{\mathcal{R}}.
\end{aligned}
\end{equation*}
On the basis of Theorem \ref{theobi-sim},
one can be concluded that $\mathcal{R}$ is a weak bisimulation relation on $\Delta_8$.
\end{example}

Naturally, when $\mathcal{R}_{\mathcal{A}}$ does not satisfy the condition in Theorem  \ref{theobi-sim},
then the most crucial task is to find the maximum weak bisimulation relation contained in it.
It is worthy noting that literature \cite{Lirui2021modelreduction} confirms the existence of the maximum weak bisimulation relation contained in the given equivalence relation, which is stated as follows.

\begin{lemma}{\rm \cite{Lirui2021modelreduction}}
The set of all weak bisimulation relations between BCN \eqref{BCNalg} and its replication system included
in $\mathcal{R}_{\mathcal{A}}$ has a unique maximal element (with respect to set inclusion),
and the maximal element, expressed by $\mathcal{E}_{\mathcal{R}_\mathcal{A}}$, is also an equivalence relation on $\Delta_N$.
\end{lemma}

In the following, we aim to calculate $\mathcal{E}_{\mathcal{R}_\mathcal{A}}$ by virtue of constructing a series of equivalence relation ${\{\mathcal{R}_k\}}_{k\in\mathbb{N}}$,
 where their matrix forms are shown as follows
 \begin{equation}\label{Sk}
 \left\{
 \begin{aligned}
     \Upsilon_{\mathcal{R}_0}&= \Upsilon_{\mathcal{R}_{\mathcal{A}}}\in\mathcal{B}^{N\times N},\\
     \Upsilon_{\mathcal{R}_{k+1}}&=\Upsilon_{\mathcal{R}_{k}}\wedge M_{\mathcal{R}_{k}, F}\in \mathcal{B}^{N\times N}, k\geq 1.
     \end{aligned}
     \right.
 \end{equation}
It is easy to see that $I_N\leq \Upsilon_{\mathcal{R}_{k+1}}\leq  \Upsilon_{\mathcal{R}_{k}}, \forall k\in\mathbb{N}$.
Since $\mathcal{B}^{N\times N}$ is a finite set of matrices,
then there exists $k^*\in\mathbb{N}$, such that
$k^*={\mathrm{arcmin}}_{k\in\mathbb{N}} \{\Upsilon_{\mathcal{R}_{k}}=\Upsilon_{\mathcal{R}_{k+1}}\}$.
On this basis, we give an approach to calculate the maximum weak bisimulation relation between BCN \eqref{BCNalg} and its replication system contained in an equivalence relation.

\begin{proposition}\label{propomaxbi-sim}
Let $k^*={\mathrm{arcmin}}_{k\in\mathbb{N}} \{\Upsilon_{\mathcal{R}_{k}}$ $=\Upsilon_{\mathcal{R}_{k+1}}\}$,
then $\mathcal{E}_{\mathcal{R}_{\mathcal{A}}}=\mathcal{R}_{k^*}$.
\end{proposition}
\begin{proof}
From $\Upsilon_{\mathcal{R}_{k^*}}=\Upsilon_{\mathcal{R}_{k^*+1}}=\Upsilon_{\mathcal{R}_{k^*}}\wedge M_{\mathcal{R}_{k^*},F}$,
 we get that $\Upsilon_{\mathcal{R}_{k^*}}\leq M_{\mathcal{R}_{k^*},F}$.
According to Theorem \ref{theobi-sim},
it follows that
$\mathcal{R}_{k^*}$ is a weak bisimulation relation contained in $\mathcal{R}_{\mathcal{A}}$.
Then  we aim to verify that $\mathcal{R}_{k^*}$ is the maximum weak bisimulation relation contained in $\mathcal{R}_{\mathcal{A}}$.
By contradiction, suppose that $\widetilde{\mathcal{R}}\subseteq \mathcal{R}_{\mathcal{A}}$ is a weak bisimulation relation satisfying $\widetilde{\mathcal{R}}\not\subseteq \mathcal{R}_{k^*}$.
According to Theorem \ref{theobi-sim},
it follows that $\Upsilon_{\widetilde{\mathcal{R}}}\leq M_{\widetilde{\mathcal{R}},F}\leq M_{\mathcal{R}_{\mathcal{A}},F}$,
which means that $\Upsilon_{\widetilde{\mathcal{R}}}=\Upsilon_{\widetilde{\mathcal{R}}}\wedge M_{\widetilde{\mathcal{R}},F}\leq \Upsilon_{\mathcal{R}_{\mathcal{A}}} \wedge M_{\mathcal{R}_{\mathcal{A}},F}=\Upsilon_{{\mathcal{R}}_1}$.
Repeating this process, it can be obtained that $\Upsilon_{\widetilde{\mathcal{R}}}\leq \Upsilon_{{\mathcal{R}}_k}, \forall k\in \mathbb{N}$,
which implies that $\widetilde{\mathcal{R}}\subseteq \mathcal{R}_{k^*}$.
This is a contradiction.
Hence, $\mathcal{R}_{k^*}$ is the maximum weak bisimulation relation contained in $\mathcal{R}_{\mathcal{A}}$.
\end{proof}

\begin{remark}
It is noticed that the maximum weak bisimulation relation obtained by this paper
is consistent with the result of Theorem 4.2 in \cite{Lirui2021modelreduction}.
However, compared to the multiply calculation procedures in \cite{Lirui2021modelreduction} with complexity  $\mathcal{O}(N^4)$,
the complexity of this novel method is $\mathcal{O}(N^3)$.
\end{remark}

Let $N_{\mathcal{A}}$ be the number of the equivalence classes in ${\mathcal{E}}_{\mathcal{R}_{\mathcal{A}}}$.
 As noted in \cite{Lirui2021modelreduction},
the quotient system $\Sigma_{I}$
is established with the equivalence classes in ${\mathcal{E}}_{\mathcal{R}_{\mathcal{A}}}$ as the states,
and its one-step reachability matrix is expressed by
 \begin{equation}\label{quosysPsi1}
  \Psi_1^{[I]} =\Upsilon_{\mathcal{R}_I} \odot\Psi_1\odot {\Upsilon_{\mathcal{R}_I}}^{\mathrm{T}}
  \in\mathcal{B}^{N_{\mathcal{A}}\times N_{\mathcal{A}}},
 \end{equation}
where $\Upsilon_{\mathcal{R}_I}\in \mathcal{L}_{N_{\mathcal{A}}\times N}$ is obtained from
$\Upsilon_{{\mathcal{E}}_{\mathcal{R}_{\mathcal{A}}}}$ by collapsing the identical rows,
with $\mathcal{R}_I=\{(\delta_{N}^i, \delta_{N_\mathcal{A}}^j)\in \Delta_N \times \Delta_{N_\mathcal{A}} |{[\Upsilon_{\mathcal{R}_I}]}_{j,i}=1\}$.

In the following, we analyze the propagation of set stabilization between BCN \eqref{BCNalg} and its  quotient system $\Sigma_{I}$,
and design all kinds of state feedback controls for BCN \eqref{BCNalg} by analyzing quotient system $\Sigma_{I}$.

\begin{theorem}\label{theo-setstaBCNquo}
 BCN \eqref{BCNalg} is $\mathcal{A}$-stabilization,
if and only if
system $\Sigma_{I}$ is $\mathcal{A}_{I}$-stabilization,
where
$\mathcal{A}_{I}=\{\beta\in\Delta_{N_{\mathcal{A}}} | \exists \alpha\in \mathcal{A}, \ \mbox{s.t.,}\ \beta =\Upsilon_{\mathcal{R}_I} \alpha\}$.
\end{theorem}

 \begin{proof}
Denote by $x(t; x_0; {\{u(k)\}}_{k=0}^{t-1})$ $(\widetilde{x}(t; \widetilde{x}_0; {\{v(k)\}}_{k=0}^{t-1}))$ the state at the $t$-th step of BCN \eqref{BCNalg} (system $\Sigma_{I}$) with
initial state $x_0$ $(\widetilde{x}_0)$ and control sequence ${\{u(k)\}}_{k=0}^{t-1} $ $({\{v(k)\}}_{k=0}^{t-1})$.

(Sufficiency):
For any $x_0\in\Delta_N$,
there exists $\widetilde{x}_0\in\Delta_{N_{\mathcal{A}}}$,
such that $(x_0, \widetilde{x}_0)\in \mathcal{R}_I$.
By $\mathcal{A}_{I}$-stabilization of the quotient system,
then there exists an input sequence ${\{v(t)\}}_{t\in\mathbb{N}}$ and $T_{\widetilde{x}_0}$,
such that $\widetilde{x}(t; \widetilde{x}_0; {\{v(k)\}}_{k=0}^{t-1})\in \mathcal{A}_{I}, \forall t\geq T_{\widetilde{x}_0}$.
Since $\mathcal{R}_I$ is a bisimulation relation between BCN \eqref{BCNalg} and system $\Sigma_{I}$,
then one has an input sequence ${\{u(t)\}}_{t\in\mathbb{N}}$,
such that $(x(t; x_0; {\{u(k)\}}_{k=0}^{t-1}),$ $ \widetilde{x}(t; \widetilde{x}_0; {\{v(k)\}}_{k=0}^{t-1}))\in \mathcal{R}_I, \forall t\in\mathbb{N}$.
It follows from the character of $\mathcal{A}_{I}$ that $x(t; x_0; {\{u(k)\}}_{k=0}^{t-1})\in \mathcal{A}, \forall t\geq T_{\widetilde{x}_0}$.
Due to the arbitrariness of $x_0\in\Delta_N$,
one can be concluded that BCN \eqref{BCNalg} is $\mathcal{A}$-stabilization.

(Necessity):
For any $\widetilde{x}_0\in\Delta_{N_{\mathcal{A}}}$,
there exists $x_0\in\Delta_{N}$,
such that $(x_0, \widetilde{x}_0)\in \mathcal{R}_I$.
By $\mathcal{A}$-stabilization of BCN \eqref{BCNalg},
then there exists an input sequence ${\{u(t)\}}_{t\in\mathbb{N}}$ and $T_{x_0}$,
such that $x(t; x_0; {\{u(k)\}}_{k=0}^{t-1})\in {\mathcal{A}}, \forall t\geq T_{x_0}$.
Since $\mathcal{R}_I$ is a bisimulation relation between BCN \eqref{BCNalg} and system $\Sigma_{I}$,
then one has an input sequence ${\{v(t)\}}_{t\in\mathbb{N}}$,
such that $(x(t; x_0; {\{u(k)\}}_{k=0}^{t-1}), \widetilde{x}(t; \widetilde{x}_0;$ $ {\{v(k)\}}_{k=0}^{t-1}))\in \mathcal{R}_I, \forall t\in\mathbb{N}$,
which implies that $\widetilde{x}(t; \widetilde{x}_0; {\{v(k)\}}_{k=0}^{t-1})\in \mathcal{A}_{I}, \forall t\in T_{x_0}$.
Hence, we get that quotient system $\sum_I$ is $\mathcal{A}_I$-stabilization.
 \end{proof}

On the basis of Theorem \ref{theo-setstaBCNquo},
 the mandate focus
to investigate $\mathcal{A}_{I}$-stabilization of system $\Sigma_{I}$.
Denote by $\Gamma_{\mathcal{A}_{I}}\in \mathcal{B}^{N_{\mathcal{A}}\times N_{\mathcal{A}}}$ the matrix form of set $\mathcal{A}_{I}$,
where
\begin{equation}\label{GammaA}
\mathrm{Col}_i(\Gamma_{\mathcal{A}_{I}})=\left\{
    \begin{aligned}
       \delta_{N_{\mathcal{A}}}^i,\ &\mbox{if} \ \delta_{N_{\mathcal{A}}}^i \in \mathcal{A}_I,\\
       \mathbf{0}_{N_{\mathcal{A}}},\ &\mbox{otherwise}.
    \end{aligned}
    \right.
\end{equation}
By virtue of $\Psi_1^{[I]}\in {\mathcal{B}}^{N_{\mathcal{A}}\times N_{\mathcal{A}}}$, construct $\Psi_{l}^{[I]}=\mathrm{sgn}({(\Psi_{1}^{[I]})}^l)\in \mathcal{B}^{N_{\mathcal{A}}\times N_{\mathcal{A}}}, l\in\mathbb{N}$ be the $l$-th step reachable matrix,
where ${[\Psi_{l}^{[I]}]}_{i,j}$ equals that state $\delta_{N_{\mathcal{A}}}^i$ is reachable at the $l$-th step from state $\delta_{N_{\mathcal{A}}}^j$, $i,j\in[1:N_{\mathcal{A}}]$.
Since $\mathcal{B}^{N_{\mathcal{A}}\times N_{\mathcal{A}}}$ is a finite set of matrices,
then there exists $l^*\in\mathbb{N}$, such that
$l^*={\mathrm{arcmin}}_{l\in\mathbb{N}} \{\Psi_{l+1}^{[I]}\in \cup_{k=1}^{l} \Psi_{k}^{[I]}\}$.
 Let
\begin{equation}
\Psi^{[I]}=\mathrm{sgn}(\sum_{l=1}^{l^*} \Psi_{l}^{[I]})\in \mathcal{B}^{N_{\mathcal{A}}\times N_{\mathcal{A}}},
\end{equation}
be the reachability matrix of system $\Sigma_{I}$.
Then the necessary and sufficient condition for $\mathcal{A}_{I}$-stabilization of system $\Sigma_{I}$ is obtained in our previous works.
\begin{lemma}{\rm\cite{Mutiantian2023delaysynch}}\label{lemmasetstab}
System $\Sigma_{I}$ is $\mathcal{A}_{I}$-stabilization,
if and only if
\begin{equation}\label{equtheosetstab}
{\mathbf{1}}_{N_{\mathcal{A}}}^{\mathrm{T}}\odot[{(\Gamma_{\mathcal{A}_{I}}\odot \Psi_1^{[I]})}^{|\mathcal{A}_{I}|}
\Gamma_{\mathcal{A}_I}\odot\Psi^{[I]}]
={\mathbf{1}}_{N_{\mathcal{A}}}^{\mathrm{T}}.
  \end{equation}
\end{lemma}

For set ${\mathcal{A}}_I \subseteq \Delta_{N_{\mathcal{A}}}$, construct a series of set as follows
\begin{equation}\label{omegak}
\begin{aligned}
&\Omega_0^{\mathcal{A}_I}=\{\alpha \in\mathcal{A}_I | {(\Gamma_{\mathcal{A}_I}\odot \Psi_1^{[I]})}^{|{\mathcal{A}}_{I}|}\alpha \neq
\mathbf{0}_{N_{\mathcal{A}}}\}, \\
&\Omega^{{\mathcal{A}}_I}_{k+1}=\{\delta_{N_{\mathcal{A}}}^i |\mathrm{Col}_i(\Psi_1^{[I]})\wedge \mathcal{I}(\Omega^{{\mathcal{A}}_I}_{k})\neq \mathbf{0}_{N_{\mathcal{A}}}\} \setminus \Omega^{{\mathcal{A}}_I}_{k}, k\in[0:{N_{\mathcal{A}}}-|\Omega_0^{{\mathcal{A}}_I}|].
\end{aligned}
\end{equation}
Once formula \eqref{equtheosetstab} is satisfied,
there exists ${k}_I={\mathrm{arcmin}}_{k\in[0:{N_{\mathcal{A}}}-|\Omega_0^{{\mathcal{A}}_I}|]}\{ \Omega^{{\mathcal{A}}_I}_{k+1}=\emptyset\}$,
such that $\Delta_{N_{\mathcal{A}}}=\cup_{k=0}^{{k}_I} \Omega^{{\mathcal{A}}_I}_{k}$,
where $\Omega^{{\mathcal{A}}_I}_{k_1} \cap
 \Omega^{{\mathcal{A}}_I}_{k_2}=\emptyset, \forall k_1, k_2\in[0:{k}_I]$.
Then state feedback matrix $G\in \mathcal{L}_{M\times N}$ in \eqref{ualg} of BCN \eqref{BCNalg} is designed by Algorithm \ref{algoG},
where the significance is revealed by Theorem \ref{theosetstawithG}.

\begin{algorithm}[!htpb]
\renewcommand{\algorithmicrequire}{\textbf{Input:}}
\renewcommand{\algorithmicensure}{\textbf{Output:}}
\caption{Calculate state feedback matrix $G\in \mathcal{L}_{M\times N}$ based on the weak bisimulation method.}
\label{algoG}
\begin{algorithmic}[1]
\Require $F,  \Psi_1^{[I]}, \Upsilon_{\mathcal{R}_I}$ and $\Omega_k^{\mathcal{A}_I}, k\in[0:{k}_I]$.
\Ensure $G$.
\For{$i\in[1:N]$}
\For{$k\in[0:{k}_I]$}
\If{${\mathrm{Col}}_i(\Upsilon_{\mathcal{R}_I})\in \Omega_{k}^{\mathcal{A}_I}$}
\State{let $b=k-1$ if $k\geq 1$, otherwise $b=0$;}
\State{let $\Theta_i(\mathcal{A})=\{\delta_M^r | {\mathrm{Col}}_{(r-1)N+i}(\Upsilon_{\mathcal{R}_I} F)\in \Omega_{b}^{\mathcal{A}_I})\}$;}
\State{let ${\mathrm{Col}_i}(G)\in \Theta_i(\mathcal{A})$.}
\EndIf
\EndFor
\EndFor
\State{Return $G$.}
\end{algorithmic}
\end{algorithm}

\begin{theorem}\label{theosetstawithG}
BCN \eqref{BCNalg}  is $\mathcal{A}$-stabilizable under state feedback matrix $G$ with the shortest time,
 if and only if matrix $G$ is obtained by Algorithm \ref{algoG}.
\end{theorem}

\begin{proof}
For each initial state $x(0)=\delta_N^i$,
let $\omega_i:=\min\{\omega| x(t,x(0), {\{u(l)\}}_{l=0}^{t-1})\in \mathcal{A}, \forall t\geq \omega\}$ be the time of initial state $x(0)$ achieving $\mathcal{A}$-stabilization under state feedback matrix $G$.

(Sufficiency):
Suppose that matrix $G\in \mathcal{L}_{M\times N}$ is obtained by Algorithm \ref{algoG}.
For any $x(0)=\delta_N^i$,
there exists $k_i\in[0:{k}_I]$,
such that $\widetilde{x}(0)=\Upsilon_{\mathcal{R}_I}x(0)\in \Omega_{k_i}^{{\mathcal{A}}_I}$.
Let $u(0)=Gx(0):=\delta_M^{r_i}$.
From steps 4-6,
one has $\Upsilon_{\mathcal{R}_I}x(1,x(0),u(0))
={\mathrm{Col}}_{(r_i-1)N+i}(\Upsilon_{\mathcal{R}_I}F)\in \Omega_{b_i}^{{\mathcal{A}}_I}$,
where $b_i=k_i-1$ if $k_i>0$, otherwise  $b_i=0$.
Repeating the above operation by $u(t)=Gx(t), t\in\mathbb{N}$,
one has that $\Upsilon_{\mathcal{R}_I}x(k,x(0),{\{u(t)\}}_{t=0}^{k-1})\in \Omega_{0}^{{\mathcal{A}}_I}\subseteq {A}_I, \forall k\geq k_i$.
Due to $\mathcal{A}_I=\{\beta\in\Delta_{N_{\mathcal{A}}} | \exists \alpha\in \mathcal{A}, \ \mbox{s.t.,}\ \beta =\Upsilon_{\mathcal{R}_I} \alpha\}$ defined in Theorem \ref{theo-setstaBCNquo},
it follows that $x(k,x(0),{\{u(t)\}}_{t=0}^{k-1})\in \mathcal{A}, \forall k\geq k_i$.
From Definition \ref{defsetstab},
it can be concluded that BCN \eqref{BCNalg} with state feedback matrix $G$
is $\mathcal{A}$-stabilizable.

Next, it is in a position by the method of contradiction to prove matrix $G$ guaranteeing that the time of $\mathcal{A}$-stabilization in BCN \eqref{BCNalg} is the shortest.
Suppose that there exists feedback matrix $\widetilde{G}\in \mathcal{L}_{M\times N}$,
such that under it,
we can find an initial state $x(0)=\delta_N^\theta$ with the shorter time achieving $\mathcal{A}$-stabilization than the situation under matrix $G$.
First,
construct $\widetilde{\omega}_\theta:=\min\{\omega| x(t,x(0), {\{u(l)\}}_{l=0}^{t-1})\in \mathcal{A}, \forall t\geq \omega\}$ be the time of initial state $x(0)$ achieving $\mathcal{A}$-stabilization under state feedback matrix $\widetilde{G}$.
By the assumption above,
one has $\omega_\theta> \widetilde{\omega}_\theta$.
Since matrix $G$ is calculated by Algorithm 1,
it follows from step 3 that $\widetilde{x}(0):= \Upsilon_{\mathcal{R}_I} x(0)\in \Omega^{\mathcal{A}_I}_{\omega_\theta}$.
Then according to Definition \ref{defbi-simulation},
the reachability at the $\widetilde{\omega}_\theta$-step from $x(0)$ to stabilized set $\mathcal{A}$ in BCN \eqref{BCNalg}
under state feedback matrix $\widetilde{G}$
also demonstrates the reachability at the $\widetilde{\omega}_\theta$-step from $\widetilde{x}(0)$ to stabilized set $\mathcal{A}_I$ in quotient system $\sum_I$,
which implies that $\widetilde{x}(0)\in \Omega^{\mathcal{A}_I}_{\widetilde{\omega}_\theta}$.
On this basis,
one has $\widetilde{x}(0)\in \Omega^{\mathcal{A}_I}_{\omega_\theta} \cap\Omega^{\mathcal{A}_I}_{\widetilde{\omega}_\theta}\neq \emptyset$,
which contradicts condition $\Omega^{{\mathcal{A}}_I}_{k_1} \cap
 \Omega^{{\mathcal{A}}_I}_{k_2}=\emptyset, \forall k_1, k_2\in[0:{k}_I]$.
Therefore, the invalid assumption accesses the rationality that matrix $G$ ensures the shortest time to achieving $\mathcal{A}$-stabilization of BCN \eqref{BCNalg}.

(Necessity): If BCN \eqref{BCNalg} is $\mathcal{A}$-stabilizable under state feedback matrix $G$ with the shortest time,
then we prove that matrix $G$ is obtained by Algorithm \ref{algoG} only by verifying ${\mathrm{Col}}_i(G)\in \Theta_i(\mathcal{A}), \forall i\in[1:N]$.
For each initial state $x(0)=\delta_N^i, i\in[1:N]$,
$\omega_i$ as the shortest time of initial state $x(0)$ achieving $\mathcal{A}$-stabilization under state feedback matrix $G$ also reveals the reachability at the $\omega_i$-th step from $\widetilde{x}(0):= \Upsilon_{\mathcal{R}_I} x(0)$ to target set $\mathcal{A}_I$,
which follows from formula \eqref{omegak} that $\widetilde{x}(0)\in \Omega_{\omega_i}^{\mathcal{A}_I}$.
Meanwhile, stabilized set $\mathcal{A}$ is also reachable from $x(1, x(0), u(0))=Fu(0)x(0)$ at the $b_i$-step,
where $b_i=\omega_i-1$ if $\omega_i\geq 1$, otherwise $b_i=0$.
Let $u(0)=Gx(0)={\mathrm{Col}}_i(G):=\delta_M^{r_i}$.
According to the property of bisimulation relation in Definition \ref{defbi-simulation},
it determines the reachability at the $b_i$-th step from $\widetilde{x}(1):=\Upsilon_{\mathcal{R}_I} Fu(0)x(0)$ to target set $\mathcal{A}_I$,
\emph{i.e.,} $\widetilde{x}(1)=\Upsilon_{\mathcal{R}_I} Fu(0)x(0)={\mathrm{Col}}_{(r_i-1)N+i}(\Upsilon_{\mathcal{R}_I} F)\in \Omega_{b_i}^{\mathcal{A}_I}$.
Then it can be implied from step 6 in Algorithm \ref{algoG}  that ${\mathrm{Col}}_i(G)\in \Theta_i(\mathcal{A})$.
Due to the arbitrariness of $x(0)=\delta_N^i, i\in[1:N]$,
the result can be inferred that matrix $G$ is one of the outputs of Algorithm \ref{algoG}.
The proof is completed.
\end{proof}

\begin{remark}
By Algorithm \ref{algoG},
 all kinds of time-optimal state feedback controllers can be designed to achieve the set stabilization of BCN \eqref{BCNalg}.
In contrast to the results of set stabilization in \cite{Guoyuqian2015setsta} with  computational complexity $\mathcal{O}(N^4)$,
 the computational complexity of the approach in this paper is $ \mathcal{O}(N^3+{(N_{\mathcal{A}})}^4)$.
\end{remark}

\subsection{The comparison of set stabilization of BCNs via two kinds of bisimulations}

In this section,
we first propose the novel criterion of strong bisimulation relation by the constructed strong bisimulation matrix of BCN \eqref{BCNalg}.
Then the comparison of set stabilization of BCNs under two kinds of bisimulation methods is revealed.

Consider BCN \eqref{BCNalg}.
Let $\overline{F}=F W_{[N,M]}=[\overline{F}_1, \overline{F}_2, \ldots, \overline{F}_N]\in \mathcal{L}_{N\times NM}$,
where $\overline{F}_i=\overline{F} \delta_N^i\in \mathcal{L}_{N\times M}, i\in[1:N]$.
On this basis, we construct a strong bisimulation matrix of BCN \eqref{BCNalg}, ${\overline{M}}_{\mathcal{R}_{\mathcal{A}},F}\in\mathcal{B}^{N\times N}$ expressed as follows:
\begin{equation}\label{IIMRF}
    {[{\overline{M}}_{\mathcal{R}_{\mathcal{A}},F}]}_{i,j}=\left\{
    \begin{aligned}
        1,\ &\mbox{if} \ \Upsilon_{\mathcal{R}_{\mathcal{A}}}\odot\overline{F}_i=\Upsilon_{\mathcal{R}_{\mathcal{A}}}\odot\overline{F}_j,\\
        0,\ &\mbox{otherwise}.
    \end{aligned}
    \right.
\end{equation}

Then the criterion for the strong bisimulation relation can be formulated.

\begin{theorem}\label{theobi-simII}
An equivalence relation $\mathcal{R}_{\mathcal{A}}\subseteq\Delta_{N}\times \Delta_{N}$ is a strong bisimulation relation between BCN \eqref{BCNalg} and its replication system, if and only if
\begin{equation}\label{equtheobi-simII}
\Upsilon_{\mathcal{R}_{\mathcal{A}}}\leq {\overline{M}}_{\mathcal{R}_{\mathcal{A}},F}.
\end{equation}

\end{theorem}

\begin{proof}
According to Definition \ref{defbi-simulation2},
it follows that  $(Fux_0, Fu\overline{x}_0)\in \mathcal{R}_{\mathcal{A}}, \forall (x_0, \overline{x}_0)=(\delta_N^i, \delta_N^j)$ $\in \mathcal{R}_{\mathcal{A}}, \forall u\in\Delta_M$.
From Theorem \ref{theobi-sim}, it equals to ${\mathrm{Col}}_k(\Upsilon_{\mathcal{R}_{\mathcal{A}}}\odot\overline{F}_i)
={\mathrm{Col}}_k(\Upsilon_{\mathcal{R}_{\mathcal{A}}}\odot\overline{F}_j), \forall k\in[1:M]$.
Due to the arbitrariness of $(\delta_{N}^i, \delta_{N}^j)\in \mathcal{R}_{\mathcal{A}}$ and $u\in\Delta_M$,
it is confirmed that \eqref{equtheobi-simII} holds.
The proof is completed.
\end{proof}

On the basis of Theorem \ref{theobi-simII}, a new method is proposed to calculate the maximum strong bisimulation relation, expressed by $\overline{\mathcal{E}}_{\mathcal{R}_{\mathcal{A}}}$, included in $\mathcal{R}_{\mathcal{A}}$ by replacing matrix $M_{\mathcal{R}_{k}, F}$ with matrix ${\overline{M}}_{{\mathcal{R}}_{k}, F}$ in \eqref{Sk}.
Then the quotient system of BCN \eqref{BCNalg} based on the strong bisimulation method, expressed by $\Sigma_{II}$ is constructed in \cite{Lirui2021quotients}, as well as the equivalent formulation of set stabilization.

\begin{proposition}{\rm\cite{Lirui2021quotients}}\label{propo-quoII-setstab}
Consider BCN \eqref{BCNalg} and given set $\mathcal{A}$.
Suppose that $\overline{\mathcal{E}}_{\mathcal{R}_{\mathcal{A}}}$ is the maximum strong bisimulation relation contained in $\mathcal{R}_{\mathcal{A}}$, and matrix $\Upsilon_{\mathcal{R}_{II}}\in \mathcal{L}_{\overline{N}_\mathcal{A}\times N}$ is obtained from
$\Upsilon_{{\overline{\mathcal{E}}}_{\mathcal{R}_{\mathcal{A}}}}$ by collapsing the identical rows.
Then the following conditions hold:

1) BCN \eqref{BCNalg} is $\mathcal{A}$-stabilization,
if and only if
system $\Sigma_{II}$ with state transition matrix $F_{II}=[\overline{F}_1\ \cdots\ \overline{F}_M]$ is $\mathcal{A}_{II}$-stabilization,
where
\begin{equation}
\left\{
\begin{aligned}
\overline{F}_i&=\Upsilon_{\mathcal{R}_{II}} \odot F_i \odot \Upsilon_{\mathcal{R}_{II}}^\top, i\in[1:M],\\
\mathcal{A}_{II}&=\{\beta\in\Delta_{\overline{N}_{\mathcal{A}}} | \exists \alpha\in \mathcal{A}, \ \mbox{s.t.,}\ \beta =\Upsilon_{\mathcal{R}_{II}} \alpha\}.
\end{aligned}
\right.
\end{equation}

2) If system $\Sigma_{II}$
can be stabilized to set $\mathcal{A}_{II}$ via a feedback
law $(\overline{x}, t) \mapsto u(\overline{x}, t)$, then BCN \eqref{BCNalg} can be stabilized to $\mathcal{A}$ using the
feedback law $(x, t) \mapsto u(\Upsilon_{\mathcal{R}_{II}}x, t)$.
\end{proposition}

As indicated in Proposition \ref{propo-quoII-setstab},
only a special type of feedback laws of BCN \eqref{BCNalg} is obtained,
where all states in the same equivalence class of relation $\Upsilon_{\mathcal{R}_{II}}$ correspond to the same feedback control.
Inspired by this,
the site of much attention is to design all kinds of state feedback controls of BCN \eqref{BCNalg} based on the strong bisimulation method,
where the detailed procedure is expressed by Algorithm \ref{algoGII}.

\begin{algorithm}[!htpb]
\renewcommand{\algorithmicrequire}{\textbf{Input:}}
\renewcommand{\algorithmicensure}{\textbf{Output:}}
\caption{Calculate state feedback matrix $G\in \mathcal{L}_{M\times N}$ based on the strong bisimulation method.}
\label{algoGII}
\begin{algorithmic}[1]
\Require $F_{[II]}, \Upsilon_{\mathcal{R}_{II}}$.
\Ensure $G$.
\State{calculate sets $\Omega_k^{\mathcal{A}_{II}}$ by replacing
matrices $\Gamma_{\mathcal{A}_I}$ and $\Psi_1^{[I]}$ with matrices $\Gamma_{\mathcal{A}_{II}}$ and $F_{[II]}\mathbf{1}_M$ in \eqref{omegak}, respectively.}
\State{let $k_{II}:=\min_{k}\{k\in[0:{\overline{N}_{\mathcal{A}}}-|\Omega_0^{{\mathcal{A}}_{II}}|]|\Omega_{k+1}^{\mathcal{A}_{II}}=\emptyset\}$.}
\For{$i\in[1:N]$}
\For{$k\in[0:{k}_{II}]$}
\If{${\mathrm{Col}}_i(\Upsilon_{\mathcal{R}_{II}})\in \Omega_{k}^{\mathcal{A}_{II}}$}
\State{let $b=k-1$ if $k\geq 1$, otherwise $b=0$;}
\State{let $\overline{\Theta}_i(\mathcal{A})=\{\delta_M^r | {\mathrm{Col}}_{i}(\overline{F}_r\Upsilon_{\mathcal{R}_{II}})\in \Omega_{b}^{\mathcal{A}_{II}}\}$;}
\State{let ${\mathrm{Col}_i}(G)\in \overline{\Theta}_i(\mathcal{A})$.}
\EndIf
\EndFor
\EndFor
\State{Return $G$.}
\end{algorithmic}
\end{algorithm}

\begin{theorem}\label{theosetstawithGII}
 BCN \eqref{BCNalg}  is $\mathcal{A}$-stabilizable under state feedback matrix $G$ with the shortest time,
 if and only if matrix $G$ is obtained by Algorithm \ref{algoGII}.
\end{theorem}

\begin{proof}
The proof process is similar to Theorem \ref{theosetstawithG}', so it will not be presented here.
\end{proof}

\begin{figure}[t]
  \centering
  \subfigure[Diagram of weak bisimulation method for set stabilization of BCN \eqref{BCNalg}.]{%
  \resizebox*{6cm}{!}{\includegraphics{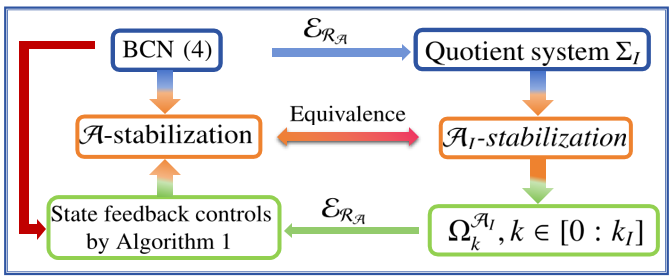}}}\hspace{0.3cm}
  \subfigure[Diagram of strong bisimulation method for set stabilization of BCN \eqref{BCNalg}.]{%
  \resizebox*{6cm}{!}{\includegraphics{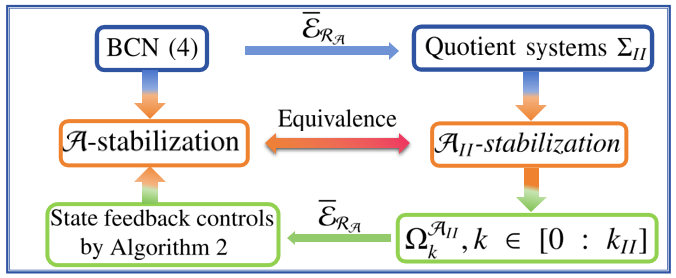}}}
\caption{Diagrams of weak bisimulation and strong bisimulation methods for set stabilization.}\label{IandII}
\end{figure}

In what follows,
the process diagrams of weak bisimulation and strong bisimulation methods for set stabilization are shown in Fig. \ref{IandII}.
And we are now in a position to elaborate the differences between two kinds of bisimulation methods in set stabilization problems, which lie in the following two aspects:
\begin{itemize}
\item The comparison of dimensionality of quotient systems induced by the weak and strong bisimulation relations.
According to Definitions \ref{defbi-simulation} and \ref{defbi-simulation2},
an equivalence relation is a strong bisimulation relation, then it must be a weak  bisimulation relation.
Hence, compared with quotient system $\Sigma_{II}$ induced by the strong bisimulation relation,
quotient system $\Sigma_I$ of  BCN \eqref{BCNalg} induced by the weak bisimulation with the smaller dimension is more easily analyzed to verify the set stabilization of BCN \eqref{BCNalg} by Theorem \ref{theo-setstaBCNquo}.

\item
The dependency relationship of the control laws on BCN \eqref{BCNalg}.
As stated in Algorithm \ref{algoG},
the feedback control via the weak bisimulation method not only need to the information of quotient system $\sum_I$, but also relies on state transition matrix $F$ of BCN \eqref{BCNalg}.
However, in Algorithm \ref{algoGII},
the feedback law based on the strong bisimulation method is only dependent on the dynamic of quotient system $\sum_{II}$, and does not require information from BCN \eqref{BCNalg}.
Therefore,
the design of state feedback control laws for BCN \eqref{BCNalg} is more easy-implement through strong bisimulation method (by Algorithm \ref{algoGII}) than weak bisimulation method (by Algorithm \ref{algoG}).

\end{itemize}

Next, we give the following example to intuitively describes the differences in handling set stabilization problems via two types of bisimulation methods.

\begin{example}
Consider again the BCN  in Example \ref{exam-comparsionIandII}.
Let $\mathcal{A}=\{\delta_8^5, \delta_8^6, \delta_8^7, \delta_8^8\}$.
By \eqref{SA}, one has $\Upsilon_{\mathcal{R}_{\mathcal{A}}}=\mathrm{diag}\{{\mathbf{1}}_{4\times 4},{\mathbf{1}}_{4\times 4}\}$.
By calculation,
we get that the maximum bisimulation-I relation contained in $\mathcal{R}_{\mathcal{A}}$ is
$\mathcal{E}_{\mathcal{R}_{\mathcal{A}}}=\{(\alpha_1, \alpha_2) | \alpha_1, \alpha_2\in\{\delta_8^1, \delta_8^2\}\}\cup\{(\alpha_3, \alpha_4) | \alpha_3, \alpha_4\in\{\delta_8^3, \delta_8^4\}\}\cup\{(\alpha_5, \alpha_6) | \alpha_5, \alpha_6\in\{\delta_8^5, \delta_8^6,\delta_8^7, \delta_8^8\}\}$.
Then the quotient system $\Sigma_{I}$ induced by $\mathcal{E}_{\mathcal{R}_{\mathcal{A}}}$ is obtained with one step transition matrix $\Psi^{[I]}=\delta_3[1\hat{+}2, 1\hat{+}3, 3]$ and $\mathcal{A}_I=\{\delta_3^3\}$.
By examine, the fact is that the quotient system $\Sigma_{I}$ is $\mathcal{A}_I$-stabilization,
which equals to the $\mathcal{A}$-stabilization of the BCN.
And based on \eqref{omegak},
we get that $\Omega^{{\mathcal{A}}_I}_0=\{\delta_3^2\}$, $\Omega^{{\mathcal{A}}_I}_1=\{\delta_3^2\}$ and $\Omega^{{\mathcal{A}}_I}_2=\{\delta_3^1\}$.
By executing Algorithm \ref{algoG},
state feedback matrix $G$ is obtained,
where
\begin{equation}\label{exam-GI}
\left\{
\begin{aligned}
{\mathrm{Col}}_1(G)&=\delta_2^2, {\mathrm{Col}}_2(G)=\delta_2^1, {\mathrm{Col}}_3(G)=\delta_2^2,\\
{\mathrm{Col}}_4(G)&=\delta_2^1,
{\mathrm{Col}}_i(G)\in\Delta_2, i\in[5:8].
\end{aligned}
\right.
\end{equation}

On the other hand,
we discuss the set stabilization of the BCN by strong bisimulation relation.
By calculation, the maximum strong bisimulation relation contained in $\mathcal{R}_{\mathcal{A}}$ is \begin{equation}
\{(\delta_8^1, \delta_8^1), (\delta_8^2, \delta_8^2), (\delta_8^3,\delta_8^3), (\delta_8^4,\delta_8^4)\}\cup \{(\alpha_5, \alpha_6) | \alpha_5, \alpha_6\in\{\delta_8^5, \delta_8^6,\delta_8^7, \delta_8^8\}\}.
\end{equation}
Then the state transition matrix of the quotient system $\Sigma_{II}$ induced by strong bisimulation relation is derived as,
\begin{equation}
F_{[II]}=\delta_5[2\ 3\ 1\ 5\ 5\ 3\ 2\ 5\ 2\ 5].
\end{equation}
By verification,
 system $\Sigma_{II}$ is $\mathcal{A}_{II}=\{\delta_5^5\}$-stabilization.
Then from Proposition \ref{propo-quoII-setstab},
it is affirmed that the BCN is $\mathcal{A}$-stabilization.
And by Algorithm \ref{algoGII}, all kinds of state feedback matrices are also designed,
where the result is similar to formula \eqref{exam-GI}.
\end{example}

\subsection{The analysis of set stabilization of PBCNs via probabilistic bisimulation}
In this section, the research scope is extended to PBCN with the algebraic form
\begin{equation}{\label{PBCNalg}}
x(t+1)=\widehat{F}_{\sigma(t)}u(t)x(t), t\in\mathbb{N},
\end{equation}
where $x(t)\in\Delta_N, u(t)\in\Delta_M$ and $\widehat{F}_{\sigma(t)}\in \mathcal{L}_{N\times NM}$, $\sigma:\mathbb{N}\mapsto [1:s]$,
and $\mathbb{P}(\sigma(t)=k)=p_k, k\in[1:s] (\sum_{k=1}^sp_k=1)$.

Denote by $x|_{x_0, \mathbf{u}}(t)$ the state at the $t$-th step of PBCN \eqref{PBCNalg} from initial state $x_0\in\Delta_N$, under a control sequence $\mathbf{u}={\{u(t)\}}_{t\in\mathbb{N}}$.
Then we give the following definition about probability bisimulation relation.

\begin{definition}{\rm \cite{Lirui2022modelreduction}}\label{defbi-Psimulation}
A relation $\mathcal{R} \subseteq \Delta_{N}\times \Delta_N$ is called a probabilistic bisimulation relation between PBCN \eqref{PBCNalg} and its replication system, if for all
$\mathcal{S}\subseteq \Delta_N$ satisfying $\mathcal{R}^{-1}\mathcal{R}(\mathcal{S})=\mathcal{S}$, we have
for any $(x_0,{\overline{x}}_0)\in \mathcal{R}$, any $\mathbf{u}={\{u(t)\}}_{t\in\mathbb{N}}$ and for any $t\in\mathbb{N}$,
\begin{equation}\label{PFuxFvy}
  \mathbb{P}(x|_{x_0,\mathbf{u}}(t)\in \mathcal{S})=\mathbb{P}(x|_{{\overline{x}}_0,\mathbf{u}}(t)\in \mathcal{R}(\mathcal{S})),
\end{equation}
where $\mathcal{R}(\mathcal{S})=\{\alpha_2 | (\alpha_1, \alpha_2)\in \mathcal{R}, \alpha_1\in \mathcal{S}\}$ and $ \mathcal{R}^{-1}\mathcal{R}(\mathcal{S})=\{\alpha_1 |(\alpha_1, \alpha_2)\in \mathcal{R}, \alpha_2\in \mathcal{R}(\mathcal{S})\}$.
\end{definition}

Next, we dedicate to expanding the applicability of the proposed method of bisimulation matrices to the probability bisimulation of PBCNs.

Consider PBCN \eqref{PBCNalg}.
Let $\widetilde{F}_{k}=\widehat{F}_k W_{[N,M]}\in \mathcal{L}_{N\times NM}, k\in[1:s]$ and $\widetilde{F}=\sum_{k=1}^s p_k\widetilde{F}_k=[\widetilde{F}_{[1]}, \widetilde{F}_{[2]}, \ldots, \widetilde{F}_{[N]}]$,
where $\widetilde{F}_{[i]}=\widetilde{F} \delta_N^i, i\in[1:N]$ and ${[\widetilde{F}_{[i]}]}_{j,q}$ is the transition probability from state $\delta_N^i$ to state $\delta_N^j$ under input $\delta_M^q, j\in[1:N], q\in[1:M]$.
Then we construct a matrix ${\widetilde{M}}_{\mathcal{R}_{\mathcal{A}},\widetilde{F}}\in\mathcal{B}^{N\times N}$,
called as the probabilistic bisimulation matrix of PBCN \eqref{PBCNalg}, as follows:
\begin{equation}\label{MRF}
    {[{\widetilde{M}}_{\mathcal{R}_{\mathcal{A}},\widetilde{F}}]}_{i,j}=\left\{
    \begin{aligned}
        1,\ &\mbox{if} \ \Upsilon_{\mathcal{R}_{\mathcal{A}}}\widetilde{F}_{[i]}
        =\Upsilon_{\mathcal{R}_{\mathcal{A}}}\widetilde{F}_{[j]},\\
        0,\ &\mbox{otherwise}.
    \end{aligned}
    \right.
\end{equation}

Then the criterion for the probabilistic bisimulation relation can be formulated.

\begin{theorem}\label{theoPbi-sim}
An equivalence relation $\mathcal{R}_{\mathcal{A}}\subseteq\Delta_{N}\times \Delta_{N}$ is a probabilistic bisimulation relation between PBCN \eqref{PBCNalg} and its replication system, if and only if
\begin{equation}\label{equtheoPbi-sim}
\Upsilon_{\mathcal{R}_{\mathcal{A}}}\leq {\widetilde{M}}_{\mathcal{R}_{\mathcal{A}},\widetilde{F}}.
\end{equation}

\end{theorem}

\begin{proof}
 (Necessity):
Assume that $\mathcal{R}_{\mathcal{A}}\subseteq\Delta_{N}\times \Delta_{N}$ is a probabilistic bisimulation relation between PBCN \eqref{PBCNalg} and its replication system.
For any $(x_0,{\overline{x}}_0)=(\delta_{N}^i, \delta_{N}^j)\in \mathcal{R}_{\mathcal{A}}$,
one has ${[\Upsilon_{\mathcal{R}_\mathcal{A}}]}_{i,j}=1$.
As $\mathcal{R}_{\mathcal{A}}$ is an equivalence relation on $\Delta_N$,
its equivalence classes are established by $[\delta_N^a]=\{\delta_N^b |(\delta_N^a, \delta_N^b)\in\mathcal{R}_{\mathcal{A}}\}, a\in[1:N]$,
and satisfy ${\mathcal{R}_{\mathcal{A}}}^{-1}\mathcal{R}_{\mathcal{A}}([\delta_N^a])=[\delta_N^a]$.
For any $a\in[1:N]$ and $\forall u=\delta_M^q\in\Delta_M$,
one has
\begin{equation}
\begin{aligned}
\mathbb{P}(x|_{x_0,u}(1)\in [\delta_N^a])&=\sum_{\delta_N^b\in[\delta_N^a]}{[\widetilde{F}_{[i]}]}_{b, q}={\mathrm{Row}}_a(\Upsilon_{\mathcal{R}_\mathcal{A}}){\mathrm{Col}}_q(\widetilde{F}_{[i]})
={[\Upsilon_{\mathcal{R}_\mathcal{A}}\widetilde{F}_{[i]}]}_{a, q},\\
 \mathbb{P}(x|_{{\overline{x}}_0,u}(1)\in [\delta_N^a])&=\sum_{\delta_N^b\in[\delta_N^a]}{[\widetilde{F}_{[i]}]}_{b, q}={\mathrm{Row}}_a(\Upsilon_{\mathcal{R}_\mathcal{A}}){\mathrm{Col}}_q(\widetilde{F}_{[j]})
={[\Upsilon_{\mathcal{R}_\mathcal{A}}\widetilde{F}_{[j]}]}_{a, q}.
\end{aligned}
\end{equation}
It follows from Definition \ref{defbi-Psimulation}
that $\mathbb{P}(x|_{x_0,u}(1)\in [\delta_N^a])=\mathbb{P}(x|_{{\overline{x}}_0,u}(1)\in [\delta_N^a])$,
$\forall a\in[1:N], u=\delta_M^q\in\Delta_M$,
\emph{i.e.,} ${\Upsilon_{\mathcal{R}_\mathcal{A}}\widetilde{F}_{[i]}}={\Upsilon_{\mathcal{R}_\mathcal{A}}\widetilde{F}_{[j]}}$.
By \eqref{MRF},
we have ${[{\widetilde{M}}_{\mathcal{R}_{\mathcal{A}}, \widetilde{F}}]}_{i,j}=1$.
Due to the arbitrariness of $(\delta_{N}^i, \delta_{N}^j)\in \mathcal{R}_{\mathcal{A}}$,
one can be concluded that \eqref{equtheobi-sim} holds.

(Sufficiency):
For any $(x_0, {\overline{x}}_0)=(\delta_{N}^i, \delta_{N}^j)\in \mathcal{R}_{\mathcal{A}}$,
 one has ${[\Upsilon_{\mathcal{R}_\mathcal{A}}]}_{i,j}=1$.
From equation \eqref{equtheobi-sim},
we get that ${[{\widetilde{M}}_{\mathcal{R}_{\mathcal{A}},\widetilde{F}}]}_{i,j}=1$,
which equals that ${\Upsilon_{\mathcal{R}_\mathcal{A}}\widetilde{F}_{[i]}}={\Upsilon_{\mathcal{R}_\mathcal{A}}\widetilde{F}_{[j]}}$.
For any $\mathcal{S}\subseteq \Delta_N$ satisfying ${\mathcal{R}_\mathcal{A}}^{-1}\mathcal{R}_\mathcal{A}(\mathcal{S})=\mathcal{S}$,
there exist $a_l, l\in[1:\overline{l}]$,
such that $\mathcal{S}=\mathcal{R}_\mathcal{A}(\mathcal{S})=\cup_{l=1}^{\overline{l}}[\delta_N^{a_l}]$.
Then for any $u(0)=\delta_M^q$,
one has
\begin{equation*}
\begin{aligned}
\mathbb{P}(x|_{x_0,u(0)}(1)\in \mathcal{S})&=\sum_{l=1}^{\overline{l}}
\mathbb{P}(x|_{x_0,u(0)}(1)\in [\delta_N^{a_l}])=
\sum_{l=1}^{\overline{l}}{[\Upsilon_{\mathcal{R}_\mathcal{A}}\widetilde{F}_{[i]}]}_{a_l, q}\\
&=\sum_{l=1}^{\overline{l}}{[\Upsilon_{\mathcal{R}_\mathcal{A}}\widetilde{F}_{[j]}]}_{a_l, q}=\sum_{l=1}^{\overline{l}}
\mathbb{P}(x|_{{\overline{x}}_0,u(0)}(1)\in [\delta_N^{a_l}])\\
&=\mathbb{P}(x|_{{\overline{x}}_0,u(0)}(1)\in \mathcal{S}).
\end{aligned}
\end{equation*}
Next, we prove through iteration that $\mathbb{P}(x|_{x_0,\mathbf{u}}(t)\in \mathcal{S})=\mathbb{P}(x|_{{\overline{x}}_0,\mathbf{u}}(t)\in \mathcal{S}), \forall t\in\mathbb{N}, \mathbf{u}={\{u(t)\}}_{t\in\mathbb{N}}$.
Suppose that the result above holds for $k\in\mathbb{N}$.
Let us check that it is also satisfied for $k+1$.
It is easy to see that
\begin{equation*}
\footnotesize
\begin{aligned}
\mathbb{P}(x|_{x_0,\mathbf{u}}(k+1)\in \mathcal{S})&=\mathbb{P}(x|_{x_k,\mathbf{u}}(1)\in \mathcal{S}| x_k=x|_{x_0,\mathbf{u}}(k)\in \mathcal{S})+\mathbb{P}(x|_{x_k,\mathbf{u}}(1)\in \mathcal{S}| x_k=x|_{x_0,\mathbf{u}}(k)\not\in \mathcal{S})\\
&=\mathbb{P}(x|_{x_k,\mathbf{u}}(1)\in \mathcal{S})\mathbb{P}(x_k=x|_{x_0,\mathbf{u}}(k)\in \mathcal{S})+\mathbb{P}(x|_{x_k,\mathbf{u}}(1)\in \mathcal{S})\mathbb{P}(x_k=x|_{x_0,\mathbf{u}}(k)\not\in \mathcal{S})\\
&=\mathbb{P}(x|_{\overline{x}_k,\mathbf{u}}(1)\in \mathcal{S})\mathbb{P}(\overline{x}_k=x|_{\overline{x}_0,\mathbf{u}}(k)\in \mathcal{S})+\mathbb{P}(x|_{\overline{x}_k,\mathbf{u}}(1)\in \mathcal{S})\mathbb{P}(\overline{x}_k=x|_{\overline{x}_0,\mathbf{u}}(k)\not\in \mathcal{S})\\
&=\mathbb{P}(x|_{\overline{x}_0,\mathbf{u}}(k+1)\in \mathcal{S}).
\end{aligned}
\end{equation*}
According to Definition \ref{defbi-Psimulation},
it shows that $\mathcal{R}_{\mathcal{A}}\subseteq\Delta_{N}\times \Delta_{N}$ is a probabilistic bisimulation relation between PBCN \eqref{PBCNalg} and its replication system.
The proof is completed.
\end{proof}

\begin{remark}
The result of Theorem  \ref{theoPbi-sim} is tantamount with Proposition 4.2  in \cite{Lirui2022modelreduction},
 and can be extended to the verification of probabilistic bisimulation relation for any relation on $\Delta_N$.
\end{remark}

Similarly, the maximum probabilistic bisimulation relation contained in equivalence relation $\mathcal{R}_{\mathcal{A}}$,
is also calculated by replacing $M_{\mathcal{R}_{k}, F}$ with $\widetilde{M}_{\mathcal{R}_{k}, \widetilde{F}}$ in \eqref{Sk}.
Then the quotient system of PBCN \eqref{PBCNalg} is obtained,
along with
the equivalent statement of the set stabilization with probability one between PBCN \eqref{PBCNalg} and its quotient system revealed by \cite{Lirui2022quotient}.

\begin{remark}
Compared with \cite{Lirui2022modelreduction} and \cite{Lirui2022quotient},
the obtained results with the probabilistic bisimulation matrix of PBCNs,
 provide the more concise forms for the verification and calculation of probabilistic bisimulations,
and establish a unified analysis framework of bisimulations.
\end{remark}

\section{An Illustrative Example}
In this section,  an example is exploited to elaborate the effectiveness of our results.
Since the analysis and control of set stabilization of BCNs based on the weak and strong bisimulation methods are generally consistent,
only the verification process by the weak bisimulation method is demonstrated here.

\begin{example}\label{exam2}
Consider the regulation network of cellular proliferation and apoptosis presented in \cite{Lirui2021modelreduction} with 10 nodes including six state nodes $\{X_i\in \mathcal{D} | i\in[1:6]\}$ and four input nodes $\{U_j\in \mathcal{D} | j\in[1:4]\}$.
The logical form is shown as follows,
\begin{equation}
\left\{
\begin{aligned}
X_1(t+1)&=U_1(t) \wedge \neg U_2(t) \wedge \neg U_3(t),\\
X_2(t+1)&=U_3(t) \wedge \neg X_1(t),\\
X_3(t+1)&=\neg X_2(t)\wedge \neg X_4(t),\\
X_4(t+1)&=U_4(t), \ X_5(t+1)=X_3(t),\ X_6(t+1)=X_4(t),
\end{aligned}
\right.
\end{equation}
where $X_i, i\in[1:6]$ represent MYC, P27, Cyc-CDK, P53, proliferation and apoptosis, respectively;
 $U_j, j\in[1:4]$ indicate the signals of growth factor, overpopulation, hypoxia and DNA Damage, separately.
Under the framework of STP, the algebraic form of BCN \eqref{BCNlog} is expressed by transition matrix $F\in \mathcal{L}_{64 \times 1024}$ with $N=64$ and $M=16$,
which is not shown due to the space limitations.
Taking $\mathcal{A}=\{\delta_{64}^i |i\in[1:4]\cup [29:36]\cup [61:64]\}$.
Next, we concern with the $\mathcal{A}$-stabilization of the BCN,
which is induced by the partical synchronization with the second, third, fourth state nodes.

\textbf{Step 1:} According to \eqref{SA},
we calculate equivalence relation $\mathcal{R}_{\mathcal{A}}=\{(\alpha_1,\alpha_2) | \{\alpha_1,$ $\alpha_2\}\subseteq \mathcal{A} \ \mbox{or}\ \{\alpha_1,$ $\alpha_2\}\subseteq {\mathcal{A}}^c\}$.
On the basis of transition matrix $F$, the one-step reachability matrix $\Psi_{1}=\mathrm{sgn}(F \mathbf{1}_{16})\in \mathcal{B}_{64\times 64}$ is derived as
\begin{equation}\label{exam2Psi_1}
    \begin{aligned}
  \Psi_{1}=[\beta_1, \beta_2, \beta_3,\ldots, \beta_{16}],
    \end{aligned}
\end{equation}
where $\beta_i\in{\mathcal{B}}_{64\times 4}$ are shown as
\begin{equation*}
    \begin{aligned}
  &\beta_1=\beta_5=\mathbf{1}_4^{\mathrm{T}}\otimes \delta_{64} [25\hat{+}29\hat{+}57\hat{+}61],\ \beta_2=\beta_6=\mathbf{1}_4^{\mathrm{T}}\otimes \delta_{64}[26\hat{+}30\hat{+}58\hat{+}62],\\
 &\beta_3=\beta_7=\mathbf{1}_4^{\mathrm{T}}\otimes \delta_{64}[27\hat{+}31\hat{+}59\hat{+}63],\ \beta_4=\beta_8=\mathbf{1}_4^{\mathrm{T}}\otimes \delta_{64}[20\hat{+}24\hat{+}52\hat{+}56],\\
 &\beta_9=\beta_{13}=\mathbf{1}_4^{\mathrm{T}}\otimes \delta_{64}[25\hat{+}29\hat{+}41\hat{+}45\hat{+}57\hat{+}61], \\ &\beta_{10}=\beta_{14}=\mathbf{1}_4^{\mathrm{T}}\otimes \delta_{64}[26\hat{+}30\hat{+}42\hat{+}46\hat{+}58\hat{+}62], \\
 &\beta_{11}=\beta_{15}=\mathbf{1}_4^{\mathrm{T}}\otimes \delta_{64}[27\hat{+}31\hat{+}43\hat{+}47\hat{+}59\hat{+}63],\\  &\beta_{12}=\beta_{16}=\mathbf{1}_4^{\mathrm{T}}\otimes \delta_{64}[20\hat{+}24\hat{+}36\hat{+}40\hat{+}52\hat{+}56].
    \end{aligned}
\end{equation*}

\textbf{Step 2:} Let $\Upsilon_{\mathcal{R}_0}=\Upsilon_{\mathcal{R}_{\mathcal{A}}}$.
By running the procedure in \eqref{Sk},
we get that $k^*={\mathrm{arcmin}}_{k\in\mathbb{N}}\{\Upsilon_{\mathcal{R}_k}=\Upsilon_{\mathcal{R}_{k+1}}\}=3$,
\emph{i.e.}, $\Upsilon_{\mathcal{R}_3}$ $\leq M_{\mathcal{R}_3, F}$.
From Proposition \ref{propomaxbi-sim},
one has $\mathcal{E}_{\mathcal{R}_{\mathcal{A}}}={\mathcal{R}_3}
=\cup_{i=1}^8{\mathcal{R}}_3^i$,
where
\begin{equation*}
\small
    \begin{aligned}
     {\mathcal{R}}_3^1&=\{(\delta_{64}^{\alpha_1},\delta_{64}^{\alpha_2}) |\{\alpha_1,\alpha_2\} \subseteq[1:4]\},\
     {\mathcal{R}}_3^2= \{(\delta_{64}^{\alpha_1},\delta_{64}^{\alpha_2}) |\{\alpha_1,\alpha_2\} \subseteq[5:12]\cup[17:28]\},\\
     {\mathcal{R}}_3^3&=\{(\delta_{64}^{\alpha_1},\delta_{64}^{\alpha_2}) |\{\alpha_1,\alpha_2\} \subseteq[13:16]\},\
     {\mathcal{R}}_3^4=\{(\delta_{64}^{\alpha_1},\delta_{64}^{\alpha_2}) |\{\alpha_1,\alpha_2\} \subseteq[29:32]\},\\
     {\mathcal{R}}_3^5&=\{(\delta_{64}^{\alpha_1},\delta_{64}^{\alpha_2}) |\{\alpha_1,\alpha_2\} \subseteq[33:36]\},\
     {\mathcal{R}}_3^6=\{(\delta_{64}^{\alpha_1},\delta_{64}^{\alpha_2}) |\{\alpha_1,\alpha_2\} \subseteq [37:44]\cup[49:60]\},\\
     {\mathcal{R}}_3^7&=\{(\delta_{64}^{\alpha_1},\delta_{64}^{\alpha_2}) |\{\alpha_1,\alpha_2\} \subseteq [45:48]\},\
     {\mathcal{R}}_3^8=\{(\delta_{64}^{\alpha_1},\delta_{64}^{\alpha_2}) |\{\alpha_1,\alpha_2\} \subseteq [61:64]\}.
    \end{aligned}
\end{equation*}
By removing identical rows from matrix $\Upsilon_{\mathcal{E}_{\mathcal{R}_{\mathcal{A}}}}$,
we gain $\Upsilon_{\mathcal{R}_I}\in \mathcal{B}_{8\times 64}$,
where
\begin{equation*}
    \begin{aligned}
&\mathrm{Row}_1(\Upsilon_{\mathcal{R}_I})
=\mathrm{Row}_1(\Upsilon_{\mathcal{E}_{\mathcal{R}_{\mathcal{A}}}}),\
\mathrm{Row}_2(\Upsilon_{\mathcal{R}_I})
=\mathrm{Row}_5(\Upsilon_{\mathcal{E}_{\mathcal{R}_{\mathcal{A}}}}),\\
&\mathrm{Row}_3(\Upsilon_{\mathcal{R}_I})
=\mathrm{Row}_{13}(\Upsilon_{\mathcal{E}_{\mathcal{R}_{\mathcal{A}}}}),\
\mathrm{Row}_4(\Upsilon_{\mathcal{R}_I})
=\mathrm{Row}_{29}(\Upsilon_{\mathcal{E}_{\mathcal{R}_{\mathcal{A}}}}),\\
&\mathrm{Row}_5(\Upsilon_{\mathcal{R}_I})
=\mathrm{Row}_{33}(\Upsilon_{\mathcal{E}_{\mathcal{R}_{\mathcal{A}}}}),\
\mathrm{Row}_6(\Upsilon_{\mathcal{R}_I})
=\mathrm{Row}_{37}(\Upsilon_{\mathcal{E}_{\mathcal{R}_{\mathcal{A}}}}),\\
&\mathrm{Row}_7(\Upsilon_{\mathcal{R}_I})
=\mathrm{Row}_{45}(\Upsilon_{\mathcal{E}_{\mathcal{R}_{\mathcal{A}}}}),\
\mathrm{Row}_8(\Upsilon_{\mathcal{R}_I})
=\mathrm{Row}_{61}(\Upsilon_{\mathcal{E}_{\mathcal{R}_{\mathcal{A}}}}).
    \end{aligned}
\end{equation*}
Replying on \eqref{quosysPsi1},
 the one-step reachability matrix of quotient system $\Sigma_I$ is established as
 \begin{equation}
     \begin{aligned}
         \Psi_1^{[I]}= \delta_8[&2\hat{+}4\hat{+}6\hat{+}7\hat{+}8,\    2\hat{+}4\hat{+}6\hat{+}7\hat{+}8,\ 2\hat{+}6,\ 2\hat{+}6,  \ 2\hat{+}4\hat{+}6\hat{+}7\hat{+}8,\\
          &2\hat{+}4\hat{+}6\hat{+}7\hat{+}8,\ 1\hat{+}2\hat{+}5\hat{+}6,\ 1\hat{+}2\hat{+}5\hat{+}6]\in \mathcal{B}_{8\times 8}.
     \end{aligned}
 \end{equation}
The state transition diagram of quotient system $\Sigma_I$ is shown in Fig. \ref{exam4},
where arrows represent one-step reachable relationships between states in $\Delta_8$.
\begin{figure}[t]
  \centering
  \includegraphics[width=10cm]{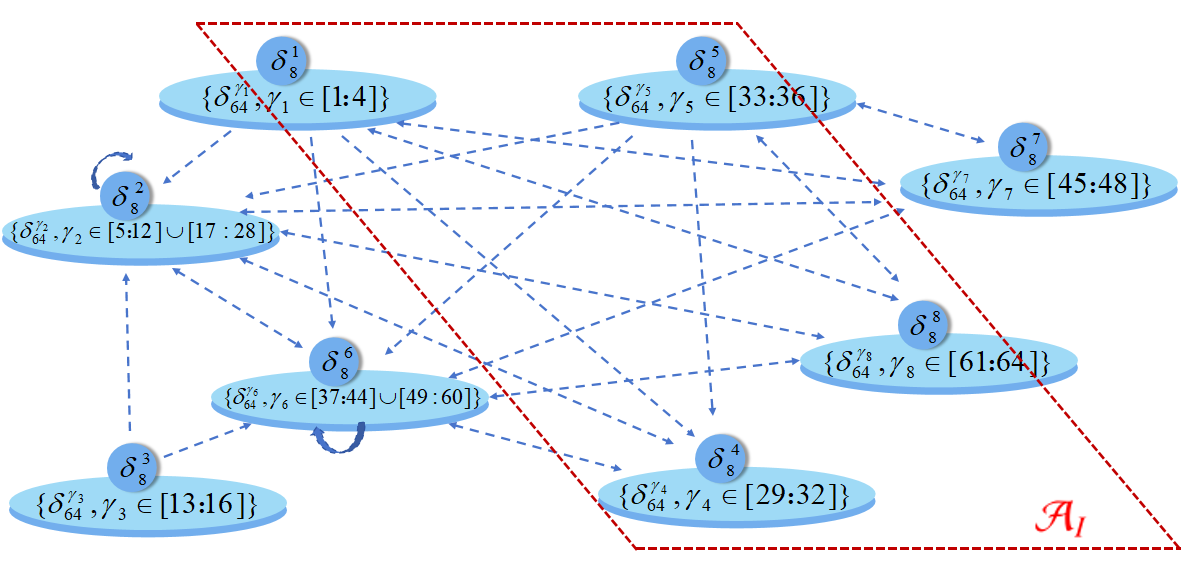}\\
  \caption{The state transition diagram of quotient system $\Sigma_I$.}\label{exam4}
\end{figure}

 According to Theorem \ref{theo-setstaBCNquo},
one can be concluded that the $\mathcal{A}$-stabilization of BCN \eqref{BCNalg} builds a equivalent bridge with the $\mathcal{A}_I$-stabilization of system $\Sigma_{I}$,
where $\mathcal{A}_I=\{\beta\in\Delta_8 |\exists \alpha \in \mathcal{A}, s.t., \beta =\Upsilon_{\mathcal{R}_I} \alpha\}=\{\delta_8^1, \delta_8^4, \delta_8^5, \delta_8^8\}$.
In what follows, we dedicate to revealing with  the $\mathcal{A}_I$-stabilization of system $\Sigma_{I}$ by verifying the condition in Lemma \ref{lemmasetstab}.

\textbf{Step 3:} For set $\mathcal{A}_I=\{\delta_8^1, \delta_8^4, \delta_8^5, \delta_8^8\}$,
we construct its matrix form $\Gamma_{\mathcal{A}_I}=[\delta_8^1,\ \mathbf{0}_8,\ \mathbf{0}_8, \delta_8^4,\  \delta_8^5,\  \mathbf{0}_8,$ $\mathbf{0}_8,\ \delta_8^8] $ by \eqref{GammaA}.
Using matrix $\Psi_1^{[I]}$,
one has $l^*={\mathrm{arcmin}}_{l\in\mathbb{N}}\{$ $\Psi_{l+1}^{[I]}\in \cup_{k=1}^{l} \Psi_{k}^{[I]}\}=3$,
 and
$\Psi^{[I]}
=\sum_{l=1}^3\Psi_{l}^{[I]}={(\mathbf{1}_{8}\otimes [1\ 1\ 0\ 1\ 1 \ 1\ 1\ 1])}^{\mathrm{T}}$.
Then one has \begin{equation}
{\mathbf{1}}_{8}^{\mathrm{T}}\odot[{(\Gamma_{\mathcal{A}_I}\odot \Psi_1^{[I]})}^{4}\Gamma_{{\mathcal{A}}_I}\odot
\Psi^{[I]}]={\mathbf{1}}_{8}^{\mathrm{T}}.
\end{equation}
Based on Theorem \ref{theo-setstaBCNquo},
one can be concluded that  system $\Sigma_{I}$ is ${\mathcal{A}}_I$-stablizable.
It also follows that BCN \eqref{BCNalg} is $\mathcal{A}$-stabilizable.
Next, we intend to design the state feedback controls to achieve the $\mathcal{A}$-stabilization of BCN \eqref{BCNalg}.

Utilizing \eqref{omegak},
we get the following sets
$\Omega_0^{{\mathcal{A}}_I}=\{\delta_8^1, \delta_8^5, \delta_8^8\},$
$\Omega_1^{{\mathcal{A}}_I}=\{\delta_8^2, \delta_8^6, \delta_8^7\},$
 $\Omega_2^{{\mathcal{A}}_I}=\{\delta_8^3, \delta_8^4\}$.
Then by Algorithm \ref{algoG},
we collect $7^{24}\times {(16)}^{8}\times 3^{24}\times 4^8$ kinds of state feedback matrices,
where the accessible set of each column is shown as follows,
\begin{equation}
    \begin{aligned}
        {\mathrm{Col}}_{i_1}(G)&\in \{\delta_{16}^{\gamma} |\gamma\in\{2,4,6,10,12,14,16\}\}, \ i_1\in[1:12]\cup[17:28],\\
        {\mathrm{Col}}_{i_2}(G)&\in \Delta_{16},\ i_2\in[13:16]\cup[29:32],\\
        {\mathrm{Col}}_{i_3}(G)&\in\{\delta_{16}^{\gamma} |\gamma\in\{4,12,16\}\},\ i_3\in[33:44]\cup[49:60],\\
       {\mathrm{Col}}_{i_4}(G)&\in \{\delta_{16}^{\gamma} |\gamma\in\{1,5,9,13\}\}, \ i_4\in[45:48]\cup[61:64].
    \end{aligned}
\end{equation}

Let us choose one of them as an example, that is
\begin{equation}\label{examG}
\begin{aligned}
    G=\delta_{16}[&\ 2\ 2\ 2\ 2\ 2\ 2\ 2\ 2\ 2\ 2\ 2\ 2\ 2\ 2\ 2\ 2\ 2\ 2\ 2\ 2\ 2\ 2\ 2\ 2\ 2\ 2\ 2\ 2\ 2\ 2\ 2\ 2\\
                  &\ 4\ 4\ 4\ 4\ 4\ 4\ 4\ 4\ 4\ 4\ 4\ 4\ 1\ 1\ 1\ 1\ 4\ 4\ 4\ 4\ 4\ 4\ 4\ 4\ 4\ 4\ 4\ 4\ 1\ 1\ 1\ 1].
\end{aligned}
\end{equation}
\begin{figure}[t]
  \centering
  \includegraphics[width=12cm]{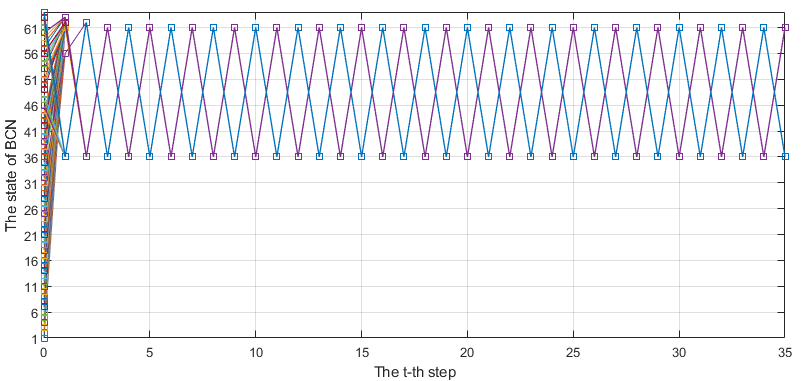}
  \caption{State trajectories of BCN \eqref{BCNalg} with state feedback matrix $G$ in \eqref{examG} from all initial states.}\label{exam2G.png}
\end{figure}

Fig. \ref{exam2G.png} shows the state trajectories of BCN \eqref{BCNalg} starting from all initial states,
where value $i\in[1:64]$ in Y-axis expresses state $\delta_{64}^i$, and X-axis represents the number of steps.
From Fig. \ref{exam2G.png}, one can be concluded that BCN \eqref{BCNalg} can achieve $\mathcal{A}$-stabilization under the state feedback control above.
\end{example}

\section{Conclusions}

A new method for set stabilization of BCNs has been proposed as a candidate for the complexity reduction, based on bisimulation relations.
In contrast to \cite{Lirui2021modelreduction}\cite{Lirui2021quotients},
novelty criteria have been given to verify two kinds of bisimulation relations by constructing the bisimulation matrices.
Meanwhile,
it has been concise and efficient to calculate the maximum bisimulation relations included by an equivalence relation,
along with the minimal quotient system.
The discrepancy about two kinds of bisimulation methods for set stabilization of BCNs has been presented,
which contains the comparison of quotient systems induced by two kinds of bisimulations,
and dependence of control laws on BCNs.
Then the proposed method can also be applied for probabilistic bisimulation relations of PBCNs,
thereby creates a unified research framework of bisimulations.
Moreover,
some problems including but not limited to reachability, observability, synchronization, output tracking of BCNs can be solved using these two bisimulation methods.
In the future, we will aim to discuss model reduction for multiple types of BCNs, combining with knowledge of bisimulation methods and graph theory, etc.

\section*{Credit author statement}
\textbf{Tiantian Mu:} Conceptualization, Formal analysis, Methodology, Software, Writing-original draft preparation, Writing-original draft.
\textbf{Jun-e Feng:} Validation, Supervision, Project administration, Funding acquisition.
\textbf{Biao Wang:} Validation, Supervision, Project administration, Funding acquisition.

\section*{Conflict of interest}
All authors declared that they have no conflict of interest relevant to this article.

\bibliographystyle{unsrt}
\bibliography{set-stabilization-reference}{}

\begin{thebibliography}{10}

\bibitem{Kauffman1969BNs}
S.~Kauffman.
\newblock Metabolic stability and epigenesis in randomly constructed genetic
  nets.
\newblock {\em Journal of Theoretical Biology}, 22(3):437--467, 1969.

\bibitem{Chengdaizhan2001morgan}
D.~Cheng.
\newblock Semi-tensor product of matrices and its application to {Morgan's}
  problem.
\newblock {\em Science China Series F: Information Sciences}, 44(3):195--212,
  2001.

\bibitem{Lifangfei2018reach}
F.~Li, H.~Yan, and H.~Harimi.
\newblock Single-input pinning controller design for reachability of {Boolean}
  networks.
\newblock {\em IEEE Transactions on Neural Networks and Learning Systems},
  29(7):3264--3269, 2018.

\bibitem{Liangjinling2017control}
J.~Liang, H.~Chen, and J.~Lam.
\newblock An improved criterion for controllability of {Boolean} control
  networks.
\newblock {\em IEEE Transactions on Automatic Control}, 62(11):6012--6018,
  2017.

\bibitem{Guoyuqian2018obser}
Y.~Guo.
\newblock Observability of {Boolean} control networks using parallel extension
  and set reachability.
\newblock {\em IEEE Transactions on Neural Networks and Learning Systems},
  29(12):6402--6408, 2018.

\bibitem{Liyifeng2023obser}
Y.~Li and J.~Zhu.
\newblock Observability decomposition of {Boolean} control networks.
\newblock {\em IEEE Transactions on Automatic Control}, 68(2):1267--1274, 2023.

\bibitem{Zhangkuizeobser2023}
K.~Zhang.
\newblock A survey on observability of {Boolean} control networks.
\newblock {\em Control Theory and Technology}, 21(2):115--147, 2023.

\bibitem{Wangyong2024setstab}
Y.~Wang, J.~Zhong, Q.~Pan, and N.~Li.
\newblock Minimal pinning control for set stability of {Boolean} networks.
\newblock {\em Applied Mathematics and Computation}, 465:128433:1--128433:12,
  2024.

\bibitem{Zhangqiliang2020setstab}
Q.~Zhang, J.~Feng, B.~Wang, and P.~Wang.
\newblock Event-triggered mechanism of designing set stabilization state
  feedback controller for switched {Boolean} networks.
\newblock {\em Applied Mathematics and Computation}, 383:125372:1--125372:10,
  2020.

\bibitem{Zhanganguo2021track}
Z.~Zhang, L.~Li, Y.~Li, and J.~Lu.
\newblock Finite-time output tracking of probabilistic {Boolean} control
  networks.
\newblock {\em Applied Mathematics and Computation}, 411:126413:1--126413:13,
  2021.

\bibitem{Zhangx2020outputtrack}
X.~Zhang, Y.~Wang, and D.~Cheng.
\newblock Output tracking of {Boolean} control networks.
\newblock {\em IEEE Transactions on Automatic Control}, 65(6):2730--2735, 2020.

\bibitem{Yangj2020synch}
J.~Yang, J.~Lu, J.~Lou, and Y.~Liu.
\newblock Synchronization of drive-response {Boolean} control networks with
  impulsive disturbances.
\newblock {\em Applied Mathematics and Computation}, 364:124679:1--124679:9,
  2020.

\bibitem{Guoyuqian2015setsta}
Y.~Guo, P.~Wang, W.~Gui, and C.~Yang.
\newblock Set stability and set stabilization of {Boolean} control networks
  based on invariant subsets.
\newblock {\em Automatica}, 61:106--112, 2015.

\bibitem{Lifangfei2017setsta}
F.~Li and Y.~Tang.
\newblock Set stabilization for switched {Boolean} control networks.
\newblock {\em Automatica}, 78:223--230, 2017.

\bibitem{Liurongjian2017setsta}
R.~Liu, J.~Lu, J.~Lou, A.~Alsaedi, and F.~Alsaadi.
\newblock Set stabilization of {Boolean} networks under pinning control
  strategy.
\newblock {\em Neurocomputing}, 260:142--148, 2017.

\bibitem{Lixiaodong2020setsta}
X.~Li, H.~Li, Y.~Li, and X.~Yang.
\newblock Function perturbation impact on stability in distribution of
  probabilistic {Boolean} networks.
\newblock {\em Mathematics and Computers in Simulation}, 177:1--12, 2020.

\bibitem{Zhushiyong2020setsta}
S.~Zhu, Y.~Liu, J.~Lou, J.~Lu, and F.~Alsaadi.
\newblock Sampled-data state feedback control for the set stabilization of
  {Boolean} control networks.
\newblock {\em IEEE Transactions on Systems, Man, and Cybernetics: Systems},
  50(4):1580--1589, 2020.

\bibitem{Chenbingquan2020setsta}
B.~Chen, J.~Cao, and L.~Rutkowski.
\newblock Lyapunov functions for the set stability and the synchronization of
  {Boolean} control networks.
\newblock {\em IEEE Transactions on Circuits and Systems-II: Express Briefs},
  67(11):2537--2541, 2020.

\bibitem{Duleihao2022setsta}
L.~Du, Z.~Zhang, and C.~Xia.
\newblock A state-flipped approach to complete synchronization of {Boolean}
  networks.
\newblock {\em Applied Mathematics and Computation}, 443:127788:1--127788:14,
  2022.

\bibitem{Chenbing2023eventtrigg}
B.~Chen, J.~Cao, G.~Lu, and L.~Rutkowski.
\newblock Stabilization of {Markovian} jump {Boolean} control networks via
  event-triggered control.
\newblock {\em IEEE Transactions on Automatic Control}, 68(2):1215--1222, 2023.

\bibitem{Yangjian2024setsta}
J.~Yang, S.~Zhang, J.~Lou, and J.~Lu.
\newblock Finite-time set stabilization of probabilistic {Boolean} control
  networks via output-feedback control.
\newblock {\em Neurocomputing}, 572:127208:1--127208:8, 2024.

\bibitem{Linlin2024setsta}
L.~Lin, J.~Cao, J.~Lu, and L.~Rutkowski.
\newblock Set stabilization of large-scale stochastic {Boolean} networks: {A}
  distributed control strategy.
\newblock {\em IEEE/CAA Journal of Automatica Sinica}, 11(3):806--808, 2024.

\bibitem{Lij2022modelreduction}
J.~Li and P.~Stinis.
\newblock Model reduction for a power grid model.
\newblock {\em Journal of Computational Dynamics}, 9(1):1--26, 2022.

\bibitem{AresdeParga2023modelreduction}
S.~Parga, J.~Bravoa, J.~Hern¨¢ndeza, R.~Zorrillaa, and R.~Rossia.
\newblock Hyper-reduction for {Petrov-Galerkin} reduced order models.
\newblock {\em Computer Methods in Applied Mechanics and Engineering},
  416:116298:1--116298:32, 2023.

\bibitem{Koellermeier2024modelreduction}
J.~Koellermeier, P.~Krah, J.~Reiss, and Z.~Schellin.
\newblock Model order reduction for the {1D Boltzmann-BGK equation}:
  identifying intrinsic variables using neural networks.
\newblock {\em Microfluidics and Nanofluidics}, 28(16):1--24, 2024.

\bibitem{Veliz-Cuba2014modelreduction}
A.~Veliz-Cuba, B.~Aguilar, F.~Hinkelmann, and R.~Laubenbacher.
\newblock Steady state analysis of {Boolean} molecular network models via model
  reduction and computational algebra.
\newblock {\em BMC Bioinformatics}, 15(1):221:1--221:8, 2014.

\bibitem{Motoyama2013modelreduction}
F.~Motoyama, K.~Kobayashi, and Y.~Yamashita.
\newblock Fixed point preserving model reduction of {Boolean} networks focusing
  on complement and absorption laws.
\newblock {\em IEICE Transactions on Fundamentals of Electronics,
  Communications and Computer Sciences}, E106-A(5):721--728, 2023.

\bibitem{Mengmin2016modelreduction}
M.~Meng, J.~Lam, J.~Feng, and X.~Li.
\newblock $l_1$-gain analysis and model reduction problem for {Boolean} control
  networks.
\newblock {\em Information Sciences}, 348:68--83, 2016.

\bibitem{Lirui2018modelreduction}
R.~Li, T.~Chu, and X.~Wang.
\newblock Bisimulations of {Boolean} control networks.
\newblock {\em SIAM Journal on Control and Optimization}, 56(1):388--416, 2018.

\bibitem{Lirui2021modelreduction}
R.~Li, Q.~Zhang, and T.~Chu.
\newblock Reduction and analysis of {Boolean} control networks by bisimulation.
\newblock {\em SIAM Journal on Control and Optimization}, 59(2):1033--1056,
  2021.

\bibitem{Lirui2022modelreduction}
{R. Li and Q. Zhang and T. Chu}.
\newblock Bisimulations of probabilistic {Boolean} networks.
\newblock {\em SIAM Journal on Control and Optimization}, 60(15):2631--2657,
  2022.

\bibitem{Zhangqiliang2019modelreduction}
Q.~Zhang, J.~Feng, B.~Wang, and M.~Meng.
\newblock Bi-simulations of {Boolean} control networks with impulsive effects
  and its application in controllability.
\newblock {\em Asian Journal of Control}, 21(6):2559--2568, 2019.

\bibitem{Yuw2023modelreduction}
W.~Yu, Z.~Deng, Q.~Guo, and Q.~Liu.
\newblock Bi-simulations for delayed switched {Boolean} control networks and
  its application in controllability.
\newblock {\em IET Control Theory and Applications}, 17(11):1552--1565, 2023.

\bibitem{Lirui2021quotients}
{R. Li and Q. Zhang and T. Chu}.
\newblock On quotients of {Boolean} control networks.
\newblock {\em Automatica}, 125(2):109401:1--109401:9, 2021.

\bibitem{Chengdaizhan2011b}
D.~Cheng, H.~Qi, and Z.~Li.
\newblock {\em Analysis and Control of Boolean Networks: A Semi-tensor Product
  Approach}.
\newblock London: Springer, 2011.

\bibitem{Chengdaizhan2011a}
D.~Cheng, H.~Qi, and Y.~Zhao.
\newblock {\em An Introduction to Semi-tensor Product of Matrices and Its
  Application}.
\newblock Singapore: World Scientific, 2011.

\bibitem{Mutiantian2023delaysynch}
T.~Mu, J.~Feng, B.~Wang, and S.~Zhu.
\newblock Delay synchronization of drive-response {Boolean} networks and
  {Boolean} control networks.
\newblock {\em IEEE Transactions on Control of Network Systems},
  10(2):865--874, 2023.

\bibitem{Lirui2022quotient}
{R. Li and Q. Zhang and T. Chu}.
\newblock Quotients of probabilistic {Boolean} networks.
\newblock {\em IEEE Transactions on Automatic Control}, 67(11):6240--6247,
  2022.

\end{thebibliography}

\end{document}